\documentclass[12pt]{amsart}

\usepackage{amssymb,amsbsy}
\usepackage{latexsym,color,longtable}
\usepackage{verbatim,euscript}
\usepackage{fullpage,array}
\usepackage{tikz}
\usetikzlibrary{arrows,shapes,trees}

\numberwithin{equation}{section}
\usepackage[colorlinks=true,linkcolor=blue,urlcolor=violet,citecolor=magenta]{hyperref}

\definecolor{my_color}{rgb}{0,0.5,0.5}
\definecolor{mixt}{rgb}{0.5,0.3,0.2}
\definecolor{darkgreen}{rgb}{0.09, 0.35, 0.27}
\definecolor{darkblue}{rgb}{0,0.1,0.8}
\definecolor{forest}{rgb}{0.13, 0.55, 0.13}

\catcode`\@=\active
\catcode`\@=11

\renewcommand{\@cite}[2]{[{{\bf #1}\if@tempswa , #2\fi}]}
\renewcommand{\@biblabel}[1]{[{\bf #1}]\hfill}
\catcode`\@=12

\newtheorem{thm}{Theorem}[section]
\newtheorem{utv}[thm]{Condition}
\newtheorem{lm}[thm]{Lemma}
\newtheorem{cl}[thm]{Corollary}
\newtheorem{prop}[thm]{Proposition}
\newtheorem{conj}[thm]{Conjecture}

\theoremstyle{remark}
\newtheorem{rmk}[thm]{Remark}

\theoremstyle{definition}
\newtheorem{ex}[thm]{Example}
\newtheorem{df}{Definition}

\newcommand {\ah}{{\mathfrak a}}
\newcommand {\be}{{\mathfrak b}}
\newcommand {\ce}{{\mathfrak k}}

\newcommand {\ff}{{\mathfrak f}}
\newcommand {\g}{{\mathfrak g}}
\newcommand {\h}{{\mathfrak h}}
\newcommand {\ka}{{\mathfrak c}}
\newcommand {\el}{{\mathfrak l}}
\newcommand {\tel}{\tilde{\mathfrak l}}
\newcommand {\n}{{\mathfrak n}}

\newcommand {\p}{{\mathfrak p}}
\newcommand {\q}{{\mathfrak q}}
\newcommand {\rr}{{\mathfrak r}}
\newcommand {\es}{{\mathfrak s}}
\newcommand {\te}{{\mathfrak t}}
\newcommand {\ut}{{\mathfrak u}}
\newcommand {\z}{{\mathfrak z}}

\newcommand {\slno}{{\mathfrak {sl}}_{n+1}}
\newcommand {\sltn}{{\mathfrak {sl}}_{2n}}

\newcommand {\glv}{{\mathfrak {gl}}(\BV)}

\newcommand {\spn}{{\mathfrak {sp}}_{2n}}

\newcommand {\sono}{{\mathfrak {so}}_{2n+1}}
\newcommand {\sone}{{\mathfrak {so}}_{2n}}
\newcommand {\son}{{\mathfrak {so}}_{n}}

\newcommand {\eus}{\EuScript}

\newcommand {\gS}{{\eus S}}
\newcommand {\gZ}{{\eus Z}}
\newcommand {\esi}{\varepsilon}
\newcommand {\ap}{\alpha}

\newcommand {\lb}{\lambda}
\newcommand {\vp}{\varphi}

\newcommand {\ca}{{\mathcal A}}
\newcommand {\cb}{{\mathcal B}}

\newcommand {\cF}{{\mathcal F}}

\newcommand {\BV}{{\mathbb{V}}}

\newcommand {\BC}{{\mathbb C}}

\newcommand {\BQ}{{\mathbb Q}}
\newcommand {\BZ}{{\mathbb Z}}

\newcommand {\md}{/\!\!/}

\newcommand {\ads}{ {\mathrm{ad}}^* }

\newcommand {\codim}{{\mathrm{codim\,}}}

\newcommand {\ind}{{\mathrm{ind\,}}}
\newcommand {\Lie}{{\mathrm{Lie\,}}}

\newcommand {\Ima}{{\mathsf{Im\,}}}

\newcommand {\rk}{{\mathsf{rk\,}}}
\newcommand {\spe}{{\mathsf{Spec\,}}}
\newcommand {\supp}{{\mathsf{supp}}}

\newcommand {\trdeg}{{\mathrm{trdeg\,}}}

\newcommand {\tri}{{\mathfrak{sl}}_2}
\newcommand {\sltri}{{\mathfrak{sl}}_3}
\newcommand {\GR}[2]{{\textrm{{\sf\bfseries #1}}}_{#2}}

\newcommand {\ov}{\overline}
\newcommand {\un}{\underline}

\newcommand {\beq}{\begin{equation}}
\newcommand {\eeq}{\end{equation}}

\newcommand {\nap}{\n_{\{\ap\}}}
\newcommand {\napn}{\n_{\{\ap_{n}\}}}
\newcommand {\bb}{\boldsymbol{b}}

\newcommand {\bxi}{{\boldsymbol{\zeta}}}
\newcommand {\CP}{{\sf CP}}

\newcommand{\curle}{\preccurlyeq}
\renewcommand{\le}{\leqslant}
\renewcommand{\ge}{\geqslant}
\renewcommand{\lg}{\langle}
\newcommand{\rg}{\rangle}

\begin{document}
\setlength{\parskip}{2pt plus 4pt minus 0pt}
\hfill {\scriptsize October 24, 2023} 
\vskip1ex

\title[Properties of nilradicals]%
{The Frobenius semiradical, generic stabilisers, and Poisson centre for nilradicals}
\author[Dmitri Panyushev]{Dmitri I. Panyushev}
\address{Institute for Information Transmission Problems, Moscow 127051, Russia}
\email{panyush@mccme.ru}
\begin{abstract}
Let $\g$ be a complex simple Lie algebra and $\n$ the nilradical of a parabolic subalgebra of $\g$. We 
consider some properties of the coadjoint representation of $\n$ and related algebras of invariants.
This includes {\sf (i)} the problem of existence of generic stabilisers, {\sf (ii)} a description of the Frobenius semiradical of $\n$ and the Poisson 
centre $\eus Z(\n)$ of the symmetric algebra $\gS(\n)$,  {\sf (iii)} the structure of $\gS(\n)$ as $\eus Z(\n)$-module, and {\sf (iv)} the description of square integrable (=\,quasi-reductive) nilradicals. Our main technical tools are the Kostant cascade in the set of positive roots of $\g$ and the notion of optimisation of $\n$.
\end{abstract}
\maketitle

\section{Introduction}     \label{sect:intro}
\subsection{}
Let $G$ be a simple algebraic group with $\g=\Lie G$, $\g=\ut\oplus\te\oplus\ut^-$ a fixed 
triangular decomposition, and $\be=\ut\oplus\te$ the fixed Borel subalgebra. Then $\Delta$ is the root 
system of $(\g,\te)$, $\Delta^+$ is the set of positive roots corresponding  to $\ut$, and $\theta$ is the 
highest root in $\Delta^+$. Write $U,T,B$ for the connected subgroups of $G$ corresponding to 
$\ut,\te,\be$. 

Let $P=L{\cdot}N$ be a parabolic subgroup of $G$, with the unipotent radical $N$ and a Levi subgroup 
$L$. Then $\n=\Lie N$ is the nilradical of $\p=\Lie P$. The unipotent radicals of the parabolic subgroups 
provide an important class of non-reductive groups. 
For instance, $N$ is a {\it Grosshans subgroup\/} of 
$G$~\cite[Theorem\,16.4]{gr97}. Various results on the coadjoint representation of $\n$ can be found in \cite{ag85,ag03,ooms2,CP}.
Our main goal is to elaborate on invariant-theoretic properties of the coadjoint representation $(N:\n^*)$,
but we also consider actions of some larger unipotent groups on $\n^*$. 

Without loss of generality, we may assume that 
$\p$ is {\it standard}, i.e., $\p\supset\be$. Then $\n\subset \ut$ is a sum of root spaces and  $\Delta(\n)$ 
denotes the corresponding set of positive roots. Unless otherwise stated, ``a nilradical'' (in $\g$) means 
``the nilradical of a standard parabolic subalgebra'' of $\g$. 
Let $\eus K=\{\beta_1,\dots,\beta_m\}$ 
be the {\it Kostant cascade\/} in $\Delta^+$. It is a poset, and $\beta_1=\theta$ is the unique maximal 
element of $\eus K$, see  Section~\ref{subs:MCP} for details. To each $\n$, we attach the subposet 
$\eus K(\n)=\eus K\cap\Delta(\n)$. Another ingredient is the {\it optimisation\/} of $\n$. By definition, it is
the maximal nilradical, $\tilde\n$, such that $\eus K(\n)=\eus K(\tilde\n)$. If $\n=\tilde\n$, then $\n$ is said 
to be {\it optimal}. An explicit description of $\tilde\n$ via $\eus K(\n)$ is given in Section~\ref{subs:optim}. 
Properties of the optimal nilradicals are better, and in order to  
approach arbitrary nilradicals, it is convenient to consider first the optimal ones. Roughly speaking, the 
output of this article is that to a great extent invariant-theoretic properties of 
$\n$ are determined by $\eus K(\n)$ and $\tilde\n$.

\subsection{} 
Let $\q=\Lie Q$ be a Lie algebra. As usual, $\xi\in\q^*$ is said to be {\it regular}, if the stabiliser $\q^\xi$ 
has minimal dimension. Then $\q^*_{\sf reg}$ denotes the set of all regular points and 
$\ind\q:=\dim\q^\xi$ for any $\xi\in\q^*_{\sf reg}$. If $\ind\q=0$, then $\q$ is called {\it Frobenius}. Set 
$\bb(\q):=(\dim\q+\ind\q)/2$. By definition, the {\it Frobenius semiradical\/} of $\q$ is 
$\cF(\q)=\sum_{\xi\in\q^*_{\sf reg}}\q^\xi$. Hence $\cF(\q)=0$ if and only if $\q$ is Frobenius. Clearly, 
$\cF(\q)$ is a characteristic ideal of $\q$. This notion and basic results on it
are due to A.\,Ooms~\cite{ooms,ooms2}.

The symmetric algebra of $\q$, $\gS(\q)$, is a Poisson algebra equipped with the Lie--Poisson bracket $\{\ ,\,\}$. The algebra of symmetric invariants $\gS(\q)^\q$ is the centre of $(\gS(\q), \{\ ,\,\})$, 
i.e.,
\[
  \gS(\q)^\q=\gZ(\q)=\{F\in \gS(\q)\mid \{F,x\}=0 \ \ \forall x\in \q\}=\{F\in \gS(\q)\mid \{F,P\}=0 \ \ \forall P\in \gS(\q)\} .
\]
If the group $Q$ is connected, then $\gS(\q)^\q=\BC[\q^*]^Q$, i.e., the Poisson centre $\gZ(\q)$ is also 
the algebra of $Q$-invariant polynomial functions on $\q^*$. If $\mathcal P\subset \gS(\q)$ is a 
Poisson-commutative subalgebra, then $\trdeg\mathcal P\le \bb(\q)$~\cite[0.2]{vi90} and this upper 
bound is always attained~\cite{sad}. Therefore, if $\rr\subset \q$ is a Lie subalgebra, then 
$\bb(\rr)\le \bb(\q)$. The passage $\n\leadsto\tilde\n$ has the property that 
$\bb(\n)=\bb(\tilde\n)$. This implies that $\gZ(\tilde\n)=\gS(\tilde\n)^{\tilde\n}\subset \gS(\n)$, see~\cite[Prop.\,5.5]{CP}. 

An abelian subalgebra $\ah\subset\q$ is called a {\it commutative polarisation\/} 
(=\,\CP), if $\dim \ah=\bb(\q)$. Then $\bb(\ah)=\bb(\q)$. A complete classification of the nilradicals with 
\CP\ is obtained in~\cite{CP}. In Section~\ref{subs:role}, we use Rosenlicht's theorem to provide simple 
proofs of some basic properties of \CP's.

\subsection{}    \label{subs:sgp-&-F(n)}
For $\n=\p^{\sf nil}$, let $\n^-\subset\ut^-$ be the opposite nilradical, i.e., $\Delta(\n^-)=-\Delta(\n)$. 
Then $\g=\p\oplus\n^-$. Let $\g_\gamma$ denote the roots space of $\gamma\in\Delta$ and $e_\gamma\in\g_\gamma$ a nonzero vector. Consider the space
$\ce=\bigoplus_{\beta\in\eus K(\n)}\g_{-\beta}\subset\n^-$ and  
$\bxi=\sum_{\beta\in\eus K(\n)}e_{-\beta}\in\ce$. Using the vector space isomorphism 
$\n^*=\g/\p\simeq\n^-$, one can regard $\ce$ as a subspace of $\n^*$ and $\bxi$ as an element of 
$\n^*$. We say that $\ce\subset\n^*$ is a {\it cascade subspace\/} and $\bxi\in\ce$ is a {\it cascade point}. 
As usual, $\xi\in\n^*$ is called $N$-{\it generic}, if there is an open subset $\Omega\in \n^*$ such 
that $\xi\in\Omega$ and the stabiliser $\n^\xi$ is $N$-conjugate to $\n^{\xi'}$ for any $\xi'\in\Omega$. 
Any stabiliser $\n^{\nu}$ with $\nu\in\Omega$ is said to be $N$-{\it generic}, too. 

 In Section~\ref{sect:sgp-nil}, we prove that the action $(N:\n^*)$ has 
$N$-generic stabilisers if and only if the stabiliser $\n^\bxi$ is generic (and the latter is not always the 
case!). Moreover, $N$-generic stabilisers always exist if $\n$ is optimal. If $\n$ is not optimal, then one 
can consider the linear action of the larger group $\tilde N=\exp(\tilde\n)\supset N$ on $\n^*$. We 
prove that the action $(\tilde N:\n^*)$ always has an $\tilde N$-generic stabiliser, and $\tilde\n^\bxi$ is 
such a stabiliser. Actually, the equality $\tilde\n^\bxi=\n^\bxi$ holds here. For any $\n$, we give an explicit 
formula for $\n^\bxi$ via $\eus K(\n)$ and $\tilde\n$, which shows that $\n^\bxi$ is $T$-stable. Using that 
formula and the criterion for the coadjoint representations~\cite[Cor.\,1.8(i)]{tayu1}, one easily verifies 
whether the stabiliser $\n^\bxi$ is generic in each concrete example. 
For $\GR{A}{n}$, the nilradicals having a generic stabiliser for $(N:\n^*)$ are explicitly described; while for $\GR{C}{n}$, all nilradicals have a generic stabiliser, see Section~\ref{sect:generic-sl-sp}. 
A general construction of nilradicals without generic stabilisers is also provided. 

We prove that  $\cF(\n)$ is the $\be$-stable ideal of $\n$
generated by $\n^\bxi$ (regardless of the presence of generic stabilisers). Then our formula for $\n^\bxi$ 
allows us to explicitly describe $\cF(\n)$ for $\GR{A}{n}$ and $\GR{C}{n}$. For any $\g$, we provide a 
criterion for the equality $\cF(\n)=\n$ and give the complete list of nilradicals with this property.
Another observation is that if $\n\subset \n'\subset\tilde\n$, then $\cF(\n')\subset\cF(\n)$.

\subsection{}    
Since $\n$ is $B$-stable, one can consider the algebra of $Q$-invariants $\gS(\n)^Q=\BC[\n^*]^Q$ for 
any subgroup $Q\subset B$, and we are primarily interested in the unipotent subgroups $U$ and 
$\tilde N=\exp(\tilde\n)\subset U$. In Section~\ref{sect:centre-&-factor}, we prove that 
\beq        \label{eq:intro1}
   \gS(\n)^U=\gS(\n)^{\tilde N}=\gS(\tilde\n)^U=\gS(\tilde\n)^{\tilde N}
\eeq
and this common algebra is polynomial, of Krull dimension $\#\eus K(\n)$. In particular, for any optimal 
nilradical $\tilde\n$, the Poisson centre of $\gS(\tilde\n)$ is a polynomial algebra. If $\n\ne\tilde\n$,
then $\gS(\n)^N$ does not occur in~\eqref{eq:intro1}. This algebra is not always
polynomial and its Krull dimension equals $\#\eus K(\n)+\dim(\tilde\n/\n)$. Nevertheless, $\gS(\n)^N$ 
shares many properties with algebras of invariants of reductive groups. For instance, $\gS(\n)^N$ is 
finitely generated~\cite[Lemma\,4.6]{jos77} and we prove that the affine variety 
$\n^*\md N:=\spe\gS(\n)^N$  has rational singularities.

Using results on $\eus K(\n)$ and $\n^\bxi$, we describe the nilradicals having the property that 
$\ind\n=\dim\z(\n)$, where $\z(\n)$ is the centre of $\n$. By~\cite{MW}, this property is equivalent to that 
the Lie group $N$ has a ``square integrable representation''. Therefore, such nilpotent Lie algebras are 
sometimes called ``square integrable'', see~\cite{ag85,ag03}. From a modern point of view, the square 
integrable nilradicals are precisely the "quasi-reductive" ones, see \cite{bm,DKT,MY12}. Our 
description shows that all square integrable nilradicals are metabelian. 

As $\gS(\n)^U$ is a polynomial algebra, we are interested in question whether $\gS(\n)$ is a free
$\gS(\n)^U$-module. Equivalently, when is the quotient map  $\pi:\n^*\to \n^*\md U=\spe\gS(\n)^U$
equidimensional? We prove several assertions for $U$ or (what is the same) $\tilde N$.

\textbullet\ \ If $\n$ has a \CP, then $\gS(\n)$ is a free module over $\gS(\n)^U=\gS(\n)^{\tilde N}$. This includes {\bf all} nilradicals for $\g=\slno$ or $\spn$.

\textbullet\ \ In particular, if $\n$ is the optimisation of a nilradical with \CP\ (any $\g$) or any 
optimal nilradical in $\slno$ or $\spn$,  then $\gS(\n)$ is a free module over its Poisson centre 
$\eus Z(\n)$. This also implies that in these cases the enveloping algebra $\eus U(\n)$ is a free module over its centre.

\subsection{Structure of the article} 
In Section~\ref{sect:prelim}, we recall basic facts on $\eus K$ and (optimal) nilradicals. 
In Section~\ref{sect:on-CP,-sgp-Frob_semirad}, the necessary information is gathered on generic 
stabilisers, the Frobenius semiradical, and commutative polarisations. We also include 
invariant-theoretic proofs for some properties of commutative polarisations.
Our results on generic stabilisers 
and $\cF(\n)$ for a nilradical $\n\subset\g$ are gathered in Section~\ref{sect:sgp-nil}, 
whereas Section~\ref{sect:generic-sl-sp} contains explicit results for $\g=\slno$ or $\spn$. 
In Section~\ref{sect:centre-&-factor}, we study various algebras of invariants related to the 
coadjoint representation of $\n$, and in Section~\ref{sect:square-int}, we classify the square 
integrable nilradicals. Topics related to the equidimensionality of the quotient map 
$\pi:\n^*\to \n^*\md U$ are treated in Section~\ref{sect:EQ-i-module}. The lists of cascade roots for all 
$\g$ and the Hasse diagrams of some posets $\eus K$ are presented in Appendix~\ref{sect:tables}.

\un{Main notation}. 
Throughout,  $\g=\Lie(G)$ is a simple Lie algebra. Then
\begin{itemize}
\item[--] $\be$ is a fixed Borel subalgebra of $\g$ with $\ut=[\be,\be]$;
\item[--] $\te$ is a fixed Cartan subalgebra in $\be$ and $\Delta$ is the root system of $\g$ with respect to $\te$;
\item[--] $\Delta^+$ is the set positive roots corresponding to $\ut$, and 
$\theta\in\Delta^+$ is the highest root;
\item[--] $\Pi=\{\ap_1,\dots,\ap_{\rk\g}\}$ is the set of simple roots in $\Delta^+$; 
\item[--] $\te^*_\BQ$ is the $\BQ$-vector subspace of $\te^*$ spanned by $\Delta$, and $(\ ,\, )$ is the 
positive-definite form on $\te^*_\BQ$ induced by the Killing form on $\g$; 
\item[--] If $\gamma\in\Delta$, then $\g_\gamma$ is the root space in $\g$ and 
$e_\gamma\in\g_\gamma$ is a nonzero vector;
\item[--]  If $\ka\subset\ut^\pm$ is a $\te$-stable subspace, then $\Delta(\ka)\subset \Delta^\pm$ is the set of 
roots of $\ka$;
\item[--]  $\bb(\q)=(\dim\q+\ind\q)/2$ for a Lie algebra $\q$;
\item[--]  in the explicit examples, the Vinberg--Onishchik numbering of simple roots of $\g$ is used, 
see~\cite[Table\,1]{VO}.
\end{itemize}

\section{Generalities on the cascade and nilradicals}
\label{sect:prelim}

\subsection{The root order in $\Delta^+$ and the Heisenberg subset}
\label{subs:root-order}
We identify $\Pi$ with the vertices of the Dynkin diagram of $\g$. For any $\gamma\in\Delta^+$, let 
$[\gamma:\ap]$ be the coefficient of $\ap\in\Pi$ in the expression of $\gamma$ via $\Pi$. The 
{\it support\/} of $\gamma$ is $\supp(\gamma)=\{\ap\in\Pi\mid [\gamma:\ap]\ne 0\}$. As is well known, 
$\supp(\gamma)$ is a connected subset of the Dynkin diagram. For instance, $\supp(\theta)=\Pi$
and $\supp(\ap)=\{\ap\}$. Let ``$\curle$'' denote the {\it root order\/} in $\Delta^+$, i.e., we set 
$\gamma\curle\gamma'$ if $[\gamma:\ap]\le [\gamma':\ap]$ for all $\ap\in\Pi$. Then 
$(\Delta^+,\curle)$ is a graded poset, and we write $\gamma\prec\gamma'$ if $\gamma\curle\gamma'$ and $\gamma\ne\gamma'$. 
\\ \indent
An {\it upper ideal\/} of $(\Delta^+,\curle)$ is a subset $I$ such that if $\gamma\in I$ and 
$\gamma\prec\gamma'$, then $\gamma'\in I$. Therefore, $I$ is an upper ideal of $\Delta^+$ if and only if 
$\rr=\bigoplus_{\gamma\in I} \g_\gamma$ is a $\be$-stable ideal of $\ut$, i.e., $[\be,\rr]\subset\rr$. 

For a dominant weight $\lb\in\te^*_\BQ$, set $\Delta_\lb=\{\gamma\in\Delta \mid (\lb,\gamma)=0 \}$ and
$\Delta^\pm_\lb=\Delta_\lb\cap \Delta^\pm$. Then $\Delta_\lb$ is the root system of a semisimple 
subalgebra $\g^{\perp\lb}\subset \g$ and $\Pi_\lb=\Pi\cap\Delta^+_\lb$ is the set of simple roots
in $\Delta^+_\lb$. Then
\begin{itemize} 
\item \  $\p_\lb=\g^{\perp\lb}+\be$ is a standard parabolic subalgebra of $\g$; 
\item \  the set of roots of the nilradical $\n_\lb=\p_\lb^{\sf nil}$ is \ $\Delta^+\setminus \Delta^+_\lb$. It is 
also denoted by $\Delta(\n_\lb)$.
\end{itemize}
If $\lb=\theta$, then  $\n_\theta$ is a {\it Heisenberg Lie algebra} (Heisenberg nilradical) and 
$\eus H_\theta:=\Delta(\n_\theta)$ is called the {\it Heisenberg subset\/} (of $\Delta^+$).

\subsection{The cascade poset}
\label{subs:MCP} 

The recursive construction of the Kostant cascade in $\Delta^+$ begins with $\beta_1=\theta$. On the 
next step, we take the highest roots in the irreducible subsystems of $\Delta_\theta$. These roots are 
called the {\it descendants} of $\beta_1$. The same construction is then applied to every descendant 
of $\beta_1$, and so on. This procedure eventually terminates and yields a set 
$\eus K=\{\beta_1,\beta_2,\dots,\beta_m\}\subset\Delta^+$, which is called {\it the Kostant cascade}. 
The roots in $\eus K$ are {\it strongly orthogonal}, which means that $\beta_i\pm\beta_j\not\in\Delta$ 
for all $i,j$. We make $\eus K$ a poset by letting that $\beta_i$ covers $\beta_j$ if and only if $\beta_j$ is 
a descendant of $\beta_i$. Then $\beta_1$ is the unique maximal element of $\eus K$.
We refer to~\cite[Sect\,2]{jos77}, \cite{ko12}, \cite[2.2]{CP} for more details. Let us summarise the main features of $\eus K$.

\begin{itemize}
\item $\eus K$ is a maximal set of strongly orthogonal roots in $\Delta^+$;
\item each $\beta_i$ is the highest root of the irreducible root system 
$\Delta\lg i\rg\subset\Delta$ with simple roots \ $\supp(\beta_i)$;
\item  $\eus K$ is also a subposet of $(\Delta^+,\curle)$, which provides the same 
poset structure as above;
\item  one has $\beta_j\prec \beta_i$ if and only if $\supp(\beta_j)\varsubsetneq\supp(\beta_i)$;
\item $\beta_j$ and $\beta_i$ are incomparable in $\eus K$ if and only if 
$\supp(\beta_j)\cap\supp(\beta_i)=\varnothing$; 
\item The numbering of $\eus K$ is not canonical. It is only required to be a linear extension of $(\eus K,\curle)$, i.e., if $\beta_j\prec\beta_i$, then $j>i$. In specific examples considered below we use the numbering of cascade roots given in Appendix~\ref{sect:tables}.
\end{itemize}

\noindent
Using the decomposition $\Delta^+ =\Delta^+_\theta \sqcup \eus H_\theta$ and induction on $\rk\g$, one
 obtains the disjoint union parametrised by $\eus K$: 
\beq            \label{eq:decomp-Delta}
     \Delta^+=\bigsqcup_{i=1}^m \eus H_{\beta_i} ,
\eeq
where $\eus H_{\beta_i}$ is the Heisenberg subset of $\Delta\lg i\rg^+$ and 
$\eus H_{\beta_1}=\eus H_\theta$. For $1\le i\le m$, let $\g\lg i\rg\subset \g$ be the simple Lie algebra with root system 
$\Delta\lg i\rg$. The geometric counterpart of \eqref{eq:decomp-Delta} is the vector space sum 
$\ut=\bigoplus_{i=1}^m  \h_i$, where $\h_i$ is the Heisenberg Lie algebra in $\g\lg i\rg$ and 
$\Delta(\h_i)=\eus H_{\beta_i}$. In particular, $\h_1=\n_\theta$. For each $\beta_i\in\eus K$, we set 
$\Phi(\beta_i)=\Pi\cap \eus H_{\beta_i}$. It then follows from~\eqref{eq:decomp-Delta} that 
$
   \displaystyle        \Pi=\bigsqcup_{\beta_i\in\eus K} \Phi(\beta_i) 
$.
Note that $\#\Phi(\beta_i)\le 2$ and $\#\Phi(\beta_i)= 2$ if and only if the algebra $\g\lg i\rg$ is of type 
$\GR{A}{n}$ with $n\ge 2$.
Our definition of the subsets $\Phi(\beta_i)$ yields the well-defined map  $\Phi^{-1}: \Pi\to \eus K$, 
where $\Phi^{-1}(\ap)=\beta_i$ if $\ap\in \Phi(\beta_i)$. Note that $\ap\in \supp(\Phi^{-1}(\ap))$ and
$\ap\in\Phi(\Phi^{-1}(\ap))$.
We think of the cascade poset as a triple $(\eus K, \curle, \Phi)$.
The corresponding Hasse diagrams, with subsets $\Phi(\beta_i)$ attached to every node, are 
depicted in~\cite[Section\,6]{CP}. Some of them are included in Appendix~\ref{sect:tables}. 

\noindent
Obviously, $\#\eus K\le \rk\g$, and $\#\eus K=\rk\g$ if and only if each $\beta_i$ is a multiple of a 
fundamental weight for $\g\lg i\rg$. Recall that $\theta$ is a multiple of a fundamental weight of $\g$ if 
and only if $\g$ is not of type $\GR{A}{n}$, $n\ge 2$. It is well known that $\#\eus K=\rk\g$ if and only if 
$\ind\be=0$. This happens exactly if $\g\not\in\{\GR{A}{n}, \GR{D}{2n+1}, \GR{E}{6}\}$ and then 
$\Phi^{-1}$ yields a bijection between $\eus K$ and $\Pi$.

\subsection{Nilradicals and optimal nilradicals}
\label{subs:optim}
Let $\p\supset\be$ be a standard parabolic subalgebra of $\g$, with nilradical $\n=\p^{\sf nil}$. If 
$\Pi\cap\Delta(\n)=\eus T$, then we write $\n=\n_{\eus T}$ and $\p=\p_{\eus T}$. Here $\eus T$ is the set of minimal elements 
of the poset $(\Delta(\n), \curle)$ and 
$\Pi\setminus \eus T$ is the set of simple roots for the standard Levi subalgebra $\el_{\eus T}\subset\p_{\eus T}$.
Clearly, $\eus T\ne \varnothing$ if and only if $\n_{\eus T}\ne\{0\}$.

The integer $d_{\eus T}=\sum_{\ap\in {\eus T}}[\theta:\ap]$ is the {\it depth\/} of $\n_{\eus T}$. 
Letting 
\[
  \Delta_{\eus T}(i)=\{\gamma\in \Delta^+\mid \sum_{\ap\in {\eus T}}[\gamma:\ap]=i\} \ \text{ and } \ 
\n_{\eus T}(i)=\bigoplus_{\gamma\in \Delta_{\eus T}(i)}\g_\gamma ,
\] 
one obtains the partition $\Delta(\n_{\eus T})=\bigsqcup_{i=1}^{d_{\eus T}}\Delta_{\eus T}(i)$ and the 
canonical $\BZ$-grading
\[
        \n_{\eus T}=\bigoplus_{i=1}^{d_{\eus T}}\n_{\eus T}(i) .
\]
The following is well known and easy.
\begin{lm}               \label{lm:easy-Z-grad}
If $\{\n_{\eus T}^{(i)}\}_{i\ge 1}$ denotes the lower central series of\/ $\n_{\eus T}$, then 
$\n_{\eus T}^{(i)}=\bigoplus_{j\ge i}\n_{\eus T}(j)$. The centre of\/ $\n_{\eus T}$ is\/ $\z(\n)=\n_{\eus T}(d_{\eus T})$. Hence
$\n_{\eus T}$ is abelian if and only if $d_{\eus T}=1$, i.e., ${\eus T}=\{\ap\}$ and $[\theta:\ap]=1$.
\end{lm}

Set $\eus K_{\eus T}=\eus K\cap \Delta(\n_{\eus T})$. Then $\theta=\beta_1\in \eus K_{\eus T}$ for any 
nonzero nilradical $\n_{\eus T}$. 

\begin{lm}[{\cite[Section\,2]{CP}}]   
 \label{lm:optim}
For any $\eus T\subset\Pi$, one has 
\begin{enumerate}
\item $\eus K_{\eus T}$ is an upper ideal of\/ $(\eus K,\curle)$;
\item $\eus T\subset\bigcup_{\beta_j\in \eus K_{\eus T}} \Phi(\beta_j)$ and\/ 
$\n_{\eus T}\subset\bigoplus_{\beta_j\in \eus K_{\eus T}} \h_{j}$.
\end{enumerate}
\end{lm}
\noindent
A standard parabolic subalgebra $\p_{\eus T}$  is said to be {\it optimal} if 
\[
     \eus T=\bigcup_{\beta_j\in \eus K_{\eus T}} \Phi(\beta_j).
\]
This goes back to~\cite[4.10]{jos77}, and we also apply this term to $\n_{\eus T}$. Then $\n_{\eus T}$ is optimal if and 
only if  $\n_{\eus T}=\bigoplus_{\beta_j\in \eus K_{\eus T}} \h_{j}$.  For a nonempty ${\eus T}\subset\Pi$, set 
$\tilde{\eus T}=\bigcup_{\beta_j\in \eus K_{\eus T}} \Phi(\beta_j)$ and consider the nilradical $\n_{\tilde{\eus T}}$. Then 
$\eus K_{\eus T}=\eus K_{\tilde{\eus T}}$ and 
$\n_{\eus T}\subset\n_{\tilde{\eus T}}=\bigoplus_{\beta_j\in \eus K_{\eus T}} \h_{i}$. Hence $\n_{\tilde{\eus T}}$ is optimal, 
it is the minimal optimal nilradical containing $\n_{\eus T}$, and it is the maximal element of the set
of nilradicals $\{\n'\mid \Delta(\n')\cap\eus K=\eus K_{\eus T}\}$.

\begin{df}
The nilradical $\n_{\tilde{\eus T}}$ is called the {\it optimisation} of $\n_{\eus T}$. 
\end{df}
\noindent 
If ${\eus T}\subset \Pi$ is not specified for a given nilradical $\n$, then $\eus K(\n):=\eus K\cap\Delta(\n)$ 
and we write $\tilde\n$ for the optimisation of $\n$.  

\begin{prop}[{cf.~\cite[2.4]{jos77}, \cite[2.3]{CP}}]   
\label{prop:jos}
Let\/ $\tilde\n$ be the optimisation of a nilradical\/ $\n$. Then 
\begin{itemize}
\item \ $\ind\n=\dim(\tilde\n/\n) + \# \eus K(\n)$;
\item \ $\ind\tilde\n=  \# \eus K(\tilde\n)= \# \eus K(\n)$ and 
$\bb(\n)=\bb(\tilde\n)$.
\end{itemize}
\end{prop}

\begin{rmk}     \label{rem:more}
The merit of optimisation is that the passage from $\n$ to $\tilde\n$  does not change $\eus K(\n)$ and 
$\bb(\n)$. More generally, if two nilradicals $\n'$ and $\n$ have the same optimisation, then
$\eus K(\n')=\eus K(\n)$ and $\bb(\n')=\bb(\n)$. For instance, this happens if
$\n\subset\n'\subset\tilde\n$.
\end{rmk}
\begin{ex}          \label{ex:sln}
(1) If $\g=\slno$, then $\ut$ is the set of strictly upper-triangular matrices, $\ap_i=\esi_i-\esi_{i+1}$ 
($1\le i\le n$), and $\eus K=\{\beta_1,\dots,\beta_t\}$, where $t=[(n+1)/2]$ and 
$\beta_i=\ap_i+\dots +\ap_{n+1-i}=\esi_i-\esi_{n+2-i}$. Here $(\eus K,\curle)$ is a chain and 
$\Phi(\beta_i)=\{\ap_i,\ap_{n+1-i}\}$.

(2) Take ${\eus T}=\{\ap_2,\ap_6\}$ and the nilradical $\n_{\eus T}\subset \mathfrak{sl}_7$. Then
$\eus K_{\eus T}=\{\beta_1,\beta_2\}$ and therefore $\tilde{\eus T}=\{\ap_1,\ap_2,\ap_5,\ap_6\}$, cf. the 
matrices below.
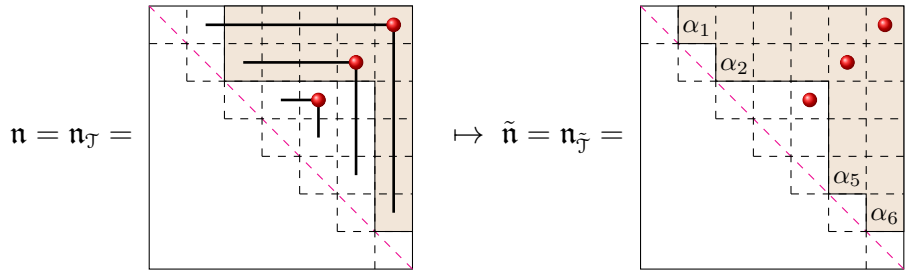
\begin{center}
\raisebox{8.5ex}{$\n=\n_{\eus T}=$} \ 
\begin{tikzpicture}[scale= .50]
\draw (0,0)  rectangle (7,7);
\draw[dashed,magenta]  (7,0) -- (0,7);
\path[draw,fill=brown!20]  (2,7) -- (7,7) -- (7,1) -- (6,1) -- (6,5)--(2,5)--cycle ;

\path[draw, line width=1pt]  (1.5,6.5)--(6.4,6.5);
\path[draw, line width=1pt]  (6.5,1.5)--(6.5,6.4);

\path[draw, line width=1pt]  (2.5,5.5)--(5.4,5.5);
\path[draw, line width=1pt]  (5.5,2.5)--(5.5,5.4);

\path[draw, line width=1pt]  (3.5,4.5)--(4.4,4.5);
\path[draw, line width=1pt]  (4.5,3.5)--(4.5,4.4);

\foreach \x in {9,11,13}  \shade[ball color=red] (\x/2,\x/2) circle (2mm);
\path[draw,dashed]  (0,6) -- (7,6); 
\path[draw,dashed]  (1,5) -- (7,5); 
\path[draw,dashed]  (2,4) -- (7,4); 
\path[draw,dashed]  (3,3) -- (7,3); 
\path[draw,dashed]  (4,2) -- (7,2); 
\path[draw,dashed]  (5,1) -- (7,1); 

\path[draw,dashed]  (6,0) -- (6,7); 
\path[draw,dashed]  (5,1) -- (5,7); 
\path[draw,dashed]  (4,2) -- (4,7); 
\path[draw,dashed]  (3,3) -- (3,7); 
\path[draw,dashed]  (2,4) -- (2,7); 
\path[draw,dashed]  (1,5) -- (1,7); 
\end{tikzpicture}
\quad \raisebox{8.5ex}{$\mapsto \ \tilde\n= \n_{\tilde{\eus T}}=$} \ 
\begin{tikzpicture}[scale= .50]
\draw (0,0)  rectangle (7,7);
\draw[dashed,magenta]  (7,0) -- (0,7);
\path[draw,fill=brown!20]  (2,6)--(1,6)--(1,7)--(7,7)--(7,1) -- (6,1) -- (6,2)--(5,2)--(5,5)--(2,5)--cycle ;

\foreach \x in {9,11,13}  \shade[ball color=red] (\x/2,\x/2) circle (2mm);
\path[draw,dashed]  (0,6) -- (7,6); 
\path[draw,dashed]  (1,5) -- (7,5); 
\path[draw,dashed]  (2,4) -- (7,4); 
\path[draw,dashed]  (3,3) -- (7,3); 
\path[draw,dashed]  (4,2) -- (7,2); 
\path[draw,dashed]  (5,1) -- (7,1); 

\path[draw,dashed]  (6,0) -- (6,7); 
\path[draw,dashed]  (5,1) -- (5,7); 
\path[draw,dashed]  (4,2) -- (4,7); 
\path[draw,dashed]  (3,3) -- (3,7); 
\path[draw,dashed]  (2,4) -- (2,7); 
\path[draw,dashed]  (1,5) -- (1,7); 
\draw (1.5,6.4)  node {\footnotesize $\ap_1$} ;
\draw (2.5,5.4)  node {\footnotesize $\ap_2$} ;
\draw (5.5,2.4)  node {\footnotesize $\ap_5$} ;
\draw (6.5,1.4)  node {\footnotesize $\ap_6$} ;
\end{tikzpicture}
\end{center}

\noindent
The cells with ball represent the cascade and the thick lines depict the Heisenberg subset attached to an 
element of the cascade. By Proposition~\ref{prop:jos}, we have $\ind\tilde\n=2$, $\ind\n=6$,  and
$\bb(\n)=\bb(\tilde\n)=10$. 
\end{ex}

\begin{ex}       \label{ex:spn}
For a square matrix $A$, let $\hat A$ denote its transpose with respect to the antidiagonal.
Choose the skew-symmetric form defining $\g=\spn\subset\sltn$ such that
\[
      \spn=\left\{  \begin{pmatrix} A & M \\ M' & -\hat A\end{pmatrix}\mid M=\hat M \ \& \ M'=\hat M'\right\} ,
\]
where $A,M,M'$ are $n\times n$ matrices. Then $\ut$ (resp. $\te$) is the set of symplectic strictly upper 
triangular (resp. diagonal) matrices. Hence 
$\te=\{{\sf diag}(\esi_1,\dots,\esi_n,-\esi_n,\dots,-\esi_1)\mid \esi_i\in\BC \}$.  Recall that
$\ap_i=\esi_i-\esi_{i+1}$ ($i<n$) and $\ap_n=2\esi_n$. Then $\eus K=\{\beta_1,\dots,\beta_n\}$ is a
chain, where $\beta_i=2\esi_i$ and $\Phi(\beta_i)=\{\ap_i\}$ for all $i$. Here 
$\napn=\{\begin{pmatrix} 0 & M \\ 0 & 0\end{pmatrix}\mid M=\hat M \}$ is the nilradical of the maximal 
parabolic subalgebra with ${\eus T}=\{\ap_n\}$. It is the only (standard) {\bf abelian} nilradical in $\spn$ 
and $\eus K\subset \Delta(\napn)$ corresponds to the antidiagonal entries of $M$.
\end{ex}

\section{Generic stabilisers, the Frobenius semiradical, and commutative polarisations}
\label{sect:on-CP,-sgp-Frob_semirad}

Let $Q$ be a connected algebraic group with $\q=\Lie Q$. If $\rho:Q\to GL(\BV)$ is a representation of
$Q$, then the corresponding $Q$-action on $\BV$ is denoted by $(Q:\BV)$. For $q\in Q$ and $v\in\BV$, 
we write $q{\cdot}v$ in place of $\rho(q)v$. Likewise, $(\q:\BV)$ corresponds to $d\rho:\q\to \glv$.

\subsection{Generic stabilisers}   
\label{subs:sgp}
Let $(Q:\BV)$ be a linear action. We say that $v\in \BV$ is $Q$-{\it generic}, if there is a dense open 
subset $\Omega\subset \BV$ such that  $v\in\Omega$ and the stabiliser $\q^x$ is $Q$-conjugate to 
$\q^v$ for any $x\in\Omega$. Then any $\q^x$ ($x\in\Omega)$ is called a $Q$-{\it generic 
stabiliser\/} for the representation $(\q:\BV)$, and we say that $(\q: \BV)$ has a $Q$-generic stabiliser. 
(One can consider similar notions for non-connected groups, for arbitrary actions of $Q$, and for 
stationary subgroups $Q^x\subset Q$, but we do need it now.) By semi-continuity of orbit dimensions, the 
set $Q$-generic points is contained in the set of $Q$-regular points
\beq                \label{eq:reg-elm} 
    \BV_{\sf reg}=\{v\in\BV\mid \dim Q{\cdot}v \ \text{is maximal}\},
\eeq
but usually, this inclusion is proper. By a result of R.W.\,Richardson~\cite{r72}, if $Q$ is reductive, then 
$Q$-generic stabilisers exist for any action of $Q$ on a smooth affine variety. But this is no longer true for 
non-reductive groups, and one of our goals is to study (the presence of) generic stabilisers for the coadjoint representation of a nilradical in $\g$.

A practical method for proving the existence of $Q$-generic points and finding $Q$-generic stabilisers 
is given by A.G.\,Elashvili~\cite[Lemma\,1]{ag72}. Let $\mathsf{T}_v(Q{\cdot}v)=\q{\cdot}v$ be the 
tangent space of the orbit $Q{\cdot}v$ at $v$ and $\BV^{\q_v}$ the fixed point subspace of $\q^v$ in 
$\BV$. Then
\beq    \label{eq:alela}
   \text{$v\in \BV$ is $Q$-generic if and only if \ $\BV=\q{\cdot}v+\BV^{\q_v}$.}
\eeq
The main case of interest for us is the coadjoint representation of $Q$, when $\BV=\q^*$. For the 
coadjoint representation, we usually skip '$Q$' from notation and refer to ``generic" and ``regular" points 
(in $\q^*$) and ``generic" stabilisers (in $\q$). 
Translating Elashvili's criterion~\eqref{eq:alela} into the setting of coadjoint representations and taking 
annihilators, one obtains the following nice formula, see~\cite[Cor.\,1.8(i)]{tayu1}.  Given $\xi\in\q^*$, the 
stabiliser $\q^\xi$ is generic (i.e., $\xi$ is a $Q$-generic point) if and only if
\beq    \label{eq:krit-sgp}
      [\q,\q^\xi]\cap \q^\xi =\{0\} .
\eeq
The reason is that $(\q{\cdot}\xi)^\perp=\q^\xi$ and $((\q^*)^{\q_\xi})^\perp=[\q,\q^\xi]$, where 
$(\cdot)^\perp$ stands for the annihilator in the dual space.

\subsection{The Frobenius semiradical}    
\label{subs:frob-semi}
For the $Q$-module $\BV=\q^*$, the set of $Q$-regular (or just "regular") points $\q^*_{\sf reg}$ consists 
of all $\xi\in\q^*$ such that the stabiliser $\q^\xi$ has the minimal possible dimension. 
If $\xi\in \q^*_{\sf reg}$, then $\ind\q:=\dim \q^\xi$ is the {\it index\/} of (a Lie algebra) $\q$.
The {\it Frobenius semiradical\/} $\cF(\q)$ of a Lie algebra $\q$ is introduced by A.\,Ooms, see~\cite{ooms,ooms2}. By definition, 
\[
          \cF(\q)=\sum_{\xi\in \q^*_{\sf reg}}\q^\xi .
\]
Obviously, $\cF(\q)$ is a characteristic ideal of $\q$, and $\cF(\q)=0$ if and only $\ind\q=0$ (i.e., $\q$ is 
a {\it Frobenius\/} Lie algebra). Note that if $(\q:\q^*)$ has a generic stabiliser, then any $Q$-generic point in the sense of Section~\ref{subs:sgp} is regular, but not vice versa.

\begin{lm}[{cf.~\cite[Prop.\,1.7]{ooms2}}]       \label{lm:sgp&F}
If\/ $(\q:\q^*)$ has a generic stabiliser and $\xi\in\q^*$ is any generic point, then
$\cF(\q)$ is the $\q$-ideal generated by the sole stabiliser $\q^\xi$.
\end{lm}
\begin{proof}
By~\cite[Lemma\,1.2]{ooms2}, if $\Psi$ is open and dense in $\q^*_{\sf reg}$, then
$\cF(\q)=\sum_{\eta\in \Psi}\q^\eta$. Applying this to the set of generic points 
$\Omega\subset\q^*_{\sf reg}$, we obtain 
$\cF(\q)=\sum_{\eta\in \Omega}\q^\eta=\sum_{g\in Q}\q^{g{\cdot}\xi}$. Clearly, the last 
sum yields the ideal of $\q$ generated by $\q^\xi$.
\end{proof}

If $\q$ is {\it quadratic\/}, 
i.e., $\q\simeq\q^*$ as $Q$-module, then $\cF(\q)=\sum_{x\in \q_{\sf reg}}\q^x$. Since $x\in\q^x$ for any
$x\in\q$, we see that here $\cF(\q)=\q$ (cf.~\cite[Theorem\,3.2]{ag03}). In 
particular, this is the case if $\q$ is reductive. Following A.\,Ooms, $\q$ is said to be {\it quasi-quadratic\/} if $\cF(\q)=\q$. Another interesting property of the functor $\cF(\cdot)$ is that
$\ind\q\le \ind\cF(\q)$ and $\cF(\cF(\q))=\cF(\q)$~\cite{ooms2}.

\subsection{Commutative polarisations}
\label{subs:CP}
If $\ah$ is an abelian subalgebra of $\q$, then $\dim\ah\le \bb(\q)$, see~\cite[0.2]{vi90} 
or~\cite[Theorem\,14]{ooms}. If $\dim\ah= \bb(\q)$, then $\ah$ is called a {\it commutative polarisation\/} 
(=\CP) of $\q$, and we say that $\q$ has a \CP. If $\ah$ is a \CP\ and also an ideal of $\q$, then it is 
called a \CP-{\it ideal}. If $\q$ is solvable and has a \CP, then it also has a \CP-ideal, 
see~\cite[Theorem\,4.1]{ag03}. More generally, a similar argument shows that if $\q$ is an ideal of a solvable Lie algebra $\rr$ and $\q$ has a \CP, then $\q$ has a \CP-ideal that is $\rr$-stable.
A standard nilradical $\n$ is an ideal of $\be$. Therefore, if $\n$ has a \CP, then it also has a \CP-ideal 
that is $\be$-stable. Henceforth, "a \CP-ideal of $\n$" means "a $\be$-stable \CP-ideal of $\n$".

Basic results on commutative polarisations are presented in~\cite{ag03}. It is also shown therein that if 
$\g$ is of type $\GR{A}{n}$ or $\GR{C}{n}$, then every nilradical in $\g$ has a \CP. A complete 
classification of the nilradicals having a \CP\ is obtained in~\cite{CP}. By Lemma~\ref{lm:easy-Z-grad}, 
$\n_{\eus T}$ is abelian if and only if ${\eus T}=\{\ap\}$ and $[\theta:\ap]=1$. 
The abelian nilradical $\nap$ play a key role in our theory.  By Theorems~3.10 
and 4.1 in~\cite{CP}, a nilradical $\n$ has a \CP\ if and only if at least one of the following two conditions 
is satisfied:
\begin{enumerate}
\item  $\n=\n_\theta=\h_1$ is the {Heisenberg nilradical}. In this case, if $\ah$ 
is any maximal abelian ideal of $\be$, then $\ah\cap\n$ is a \CP-ideal of $\n$, and vice versa.
\item  there is an abelian nilradical $\nap$ such that $\n$ is contained in $\widetilde{\nap}$, the 
optimisation of  $\nap$. (There can be several abelian nilradicals with this property, and, for a ``right'' 
choice of such $\check\ap\in\Pi$,  $\n\cap\n_{\{\check\ap\}}$ is a \CP-ideal of $\n$, cf. also 
Section~\ref{sect:EQ-i-module}).
\end{enumerate}

\noindent
If $\g$ has no parabolic subalgebras with abelian nilradicals, then the Heisenberg nilradical 
$\n_\theta=\h_1$ is the only nilradical with \CP. This happens precisely if $\g$ is of type $\GR{G}{2}$, 
$\GR{F}{4}$, $\GR{E}{8}$. 
Another result of \cite{CP} is that $\n$ has a \CP\ if and only if $\tilde\n$ has. 

\subsection{The role of commutative polarisations}
\label{subs:role}
If $\ah$ is a \CP\ of a Lie algebra $\q$, then 
\begin{enumerate}
\item \ $\cF(\q)\subset \ah$~\cite[Prop.\,20]{ooms} and thereby $\cF(\q)$ is an abelian ideal. (However, it can happen that $\cF(\q)$ is abelian, whereas $\q$ has no CP, see Example~\ref{ex:metabelian}.)
\item \ Since $\ah$ is abelian, $\bb(\ah)=\dim\ah=\bb(\q)$. Therefore,
the Poisson centre $\gZ(\q)=\gS(\q)^\q$ is contained in $\gS(\ah)$~\cite[Prop.\,5.5]{CP}. Hence, if $\ah$ is a \CP-ideal of $\q$, then $\gS(\q)^\q=\gS(\ah)^\q$.
\item \ Thus, if $\ah_1,\dots,\ah_s$ are different \CP\ in $\q$, then 
$\cF(\q)\subset \bigcap_{i=1}^s \ah_i$ and $\gS(\q)^\q\subset \gS(\bigcap_{i=1}^s \ah_i)$. 
\end{enumerate}

In many cases, the presence of a \CP\ in $\n$ allows us to quickly describe the Frobenius semiradical 
for $\n$, see Sections~\ref{subs:frob-semi-sl} and \ref{subs:frob-semi-sp}. For the reader's convenience, 
we provide an invariant-theoretic proof for two basic results on abelian subalgebras mentioned above.

\begin{prop}             \label{prop:IT-proofs}
Let $\ah$ be an abelian subalgebra of $\q$. Then
\begin{enumerate}
\item \ $\dim\ah\le \bb(\q)$;
\item \ if\/ $\dim\ah=\bb(\q)$, then $\q^\xi\subset \ah$ for any $\xi\in \q^*_{\sf reg}$, i.e., $\cF(\q)\subset\ah$.
\end{enumerate}
\end{prop}
\begin{proof}
We show that both assertions are immediate consequences of Rosenlicht's 
theorem~\cite[Chap.\,1.6]{brion}. Let $A\subset Q$ be the connected (abelian) subgroup with $\Lie A=\ah$.
For the $A$-action on $\q^*$, let $\BC(\q^*)^A$ denote the field of $A$-invariant rational functions on 
$\q^*$. 

(1) Since $\ah$ is abelian, we have $\BC[\q^*]^A=\gS(\q)^A\supset \gS(\ah)$, hence 
$\trdeg\BC(\q^*)^A\ge\dim\ah$. By Rosenlicht's theorem,
\[
  \trdeg\BC(\q^*)^A=\dim\q-\max_{\xi\in\q^*}A{\cdot}\xi=
  \dim\q-\dim\ah+\min_{\xi\in\q^*}\dim\ah^\xi .
\]
Therefore,
\[
   2\dim\ah\le \dim\q+\min_{\xi\in\q^*}\dim\ah^\xi\le \dim\q+\ind\q=2\bb(\q) .
\]
\indent
(2) If $\dim\ah=\bb(\q)$, then $\min_{\xi\in\q^*}\dim\ah^\xi=\ind\q$. For $\xi\in\q^*_{\sf reg}$, one has
$\dim\q^\xi=\ind\q$. That is, $\ah^\xi=\q^\xi\subset\ah$ for any $\xi\in\q^*_{\sf reg}$.
\end{proof}
The proof above also implies that if $\dim\ah=\bb(\q)$, then $\trdeg\BC(\q^*)^A=\dim\ah$. Therefore 
$\BC[\q^*]^A=\gS(\ah)$ and $\BC(\q^*)^A$ is the fraction field of $\gS(\ah)$.

\section{Generic stabilisers and the Frobenius semiradical for the nilradicals}     
\label{sect:sgp-nil}

\noindent
In this section, we study the Frobenius semiradical, $\cF(\n)$, of a nilradical $\n\subset\g$ and the 
existence of generic stabilisers for the coadjoint representation of $N=\exp(\n)$. We also say that 
``$\n$ has a generic stabiliser'', if the coadjoint representation $(N:\n^*)$  has. 

Let $\n=\p^{\sf nil}$ be a standard nilradical and $\n^-=(\p^-)^{\sf nil}$  the opposite nilradical, 
i.e., $\Delta(\n^-)=-\Delta(\n)$. Using the vector space sum $\g=\p\oplus\n^-$ and the 
$P$-module isomorphism $\n^*\simeq \g/\p$, we identify $\n^*$ with $\n^-$ and thereby regard $\n^-$ as 
$P$-module. In terms of $\n^-$, the $\p$-action on $\n^*$ is given by the Lie bracket in $\g$, with the 
subsequent projection to $\n^-$. In particular, the coadjoint representation $\ads_\n$ of $\n$ has the 
following interpretation. If $x\in\n$ and $\xi\in\n^-$, then $(\ads_\n x){\cdot}\xi=\mathsf{pr}_{\n^-} ([x,\xi])$, 
where $\mathsf{pr}_{\n^-}: \g\to\n^-$ is the projection with kernel $\p$. 

Recall that $\eus K(\n)=\eus K\cap \Delta(\n)$. Set 
$\ce=\bigoplus_{\beta\in\eus K(\n)}\g_{-\beta}\subset \n^-\simeq\n^*$ and
$\bxi=\sum_{\beta\in\eus K(\n)}e_{-\beta}\in\ce$. We say that $\ce$ is a {\it cascade subspace} of $\n^*$ 
and $\bxi$ is a {\it cascade point} of $\n^*$. Clearly, $\bxi$ depends on the choice of nonzero root vectors
$e_{-\beta}\in\g_{-\beta}$, but all such points $\bxi$ form a sole dense $T$-orbit in $\ce$, which is 
denoted by $\ce_0$. That is, 
$\ce_0=\{\sum_{\beta\in\eus K(\n)} a_\beta e_{-\beta} \mid a_\beta\in\BC\setminus\{0\}\}$.
It was proved by A.\,Joseph~\cite[2.4]{jos77} that upon the identification of $\n^-$ and 
$\n^*$ as $P$-modules and hence $B$-modules, $B{\cdot}\bxi$ is dense in $\n^*$ and 
$B{\cdot}\bxi\subset\n^*_{\sf reg}$. In particular, $\ce_0\subset\n^*_{\sf reg}$.

\begin{prop}    \label{thm:F-&-sgp}
Let\/ $\n$ be an arbitrary nilradical and $\bxi\in\n^*$ a cascade point. Then
\begin{itemize}
\item[\sf(i)] \ $\cF(\n)$ is the $\be$-stable ideal of\/ $\n$ generated by $\n^\bxi$;
\item[\sf(ii)] \ $(\n:\n^*)$ has a generic stabiliser if and only if the stabiliser $\n^\bxi$ is generic.
\end{itemize}
\end{prop}
\begin{proof} 
{\sf (i)} \ Since $B{\cdot}\bxi$ is dense in $\n^*_{\sf reg}$, a result of A.\,Ooms~\cite[Lemma\,1.2]{ooms2} 
implies that  $\cF(\n)=\sum_{b\in B}\n^{b{\cdot}\bxi}$. Clearly, the last sum is
the smallest $\be$-stable ideal containing $\n^\bxi$.

{\sf (ii)} \ If $\bxi$ is not generic, then $[\n,\n^{\bxi}]\cap \n^{\bxi} \ne\{0\}$ and hence 
$[\n,\n^{b{\cdot}\bxi}]\cap \n^{b{\cdot}\bxi} \ne\{0\}$ for any  $b\in B$. Therefore, neither of the points 
${b{\cdot}\bxi}$ ($b\in B$) can be generic. Since $B{\cdot}\bxi$ is dense in $\n^*$, this means that generic points cannot exist in such a case.
\end{proof}

The point of Proposition~\ref{thm:F-&-sgp}{\sf (i)} is that the existence of a  generic stabiliser for $\n$ is not
required (cf. Lemma~\ref{lm:sgp&F}). We use instead the action of $B$ on $\n^*$. 

\begin{prop}   \label{prop:optim-sgp}
If\/ $\n$ is an optimal nilradical, then 
\begin{itemize}
\item[\sf (i)] \ $\n^\bxi$ is a generic stabiliser for $(\n:\n^*)$;
\item[\sf (ii)] \ the $N$-saturation of $\ce$ is dense in $\n^*$, i.e., $\ov{N{\cdot}\ce}=\n^*$.
\end{itemize}
\end{prop}
\begin{proof}
{\sf (i)} \ For an optimal nilradical $\n$, one has  $\n^\bxi=\bigoplus_{\beta\in\eus K(\n)}\g_\beta$~\cite[2.4]{jos77}.
Since the roots in $\eus K(\n)$ are strongly orthogonal, condition~\eqref{eq:krit-sgp} is satisfied for
$\xi=\bxi$.

{\sf (ii)} \ Recall that $\ind\n=\#\eus K(\n)=\dim\ce$ and $\codim_{\n^*}(\n{\cdot}\xi)=\ind\n$ for any 
$\xi\in \ce_0$. Furthermore, $\n{\cdot}\xi\cap\ce=\{0\}$ for any $\xi\in \ce$. (Use again the strong orthogonality.) Hence $\n{\cdot}\xi\oplus\ce=\n^*$ for any $\xi\in \ce_0$, which implies that 
$N{\cdot}\ce_0$ is open and dense in $\n^*$, cf. \cite[Lemma\,1.1]{ag72}.
\end{proof}

Recall that $\tilde\n$ denotes the optimisation of $\n$. Then $\eus K(\n)=\eus K(\tilde\n)$, 
$\ind\tilde\n=\#\eus K(\n)$, and 
\beq        \label{eq:compare-ind}
     \ind \n=\dim(\tilde\n/\n)+\ind\tilde\n= \dim(\tilde\n/\n)+\# \eus K(\n) .
\eeq
The cascade subspaces and cascade points associated with $\n$ or $\tilde\n$ are the same, if regarded 
as objects in $\ut^-$. But one has to distinguish them as objects in $\n^*$ or $\tilde\n^*$. For this reason, 
working simultaneously with $\n$ and $\tilde\n$, we write $\tilde\bxi$ for a cascade point in the 
cascade subspace $\tilde\ce\subset \tilde\n^*$. Note that, for the natural projection 
$\tau: \tilde\n^*\to \n^*$, one has $\tau(\tilde\ce)=\ce$ and $\tau(\tilde\bxi)=\bxi$. Since $\tau$ is 
$B$-equivariant, this allows us to transfer some good properties from the coadjoint action 
$(\tilde N:\tilde\n^*)$ to the coadjoint action $(N: \n^*)$.

If $\n$ is not optimal, then it may or may not have a generic stabiliser.
To see this, we provide an explicit description of $\n^\bxi$ for an arbitrary nilradical $\n$, which generalises Joseph's description for the optimal nilradicals.

Because $\bxi\in \n^*_{\sf reg}$ and $\tilde\bxi\in \tilde\n^*_{\sf reg}$, it follows 
from~\eqref{eq:compare-ind} that $\dim\n^\bxi=\dim(\tilde\n/\n)+\dim\tilde\n^{\tilde\bxi}$. Since $\tau$ is 
$B$-equivariant, we have $\n^\bxi\supset \tilde\n^{\tilde\bxi}=\bigoplus_{\beta\in\eus K(\n)}\g_\beta$.
If $\eus K(\n)=\{\beta_1,\dots, \beta_k\}$, then $\n\subset \tilde\n=\bigoplus_{j=1}^k\h_j$ and 
$\n=\bigoplus_{j=1}^k (\h_j\cap\n)$, see Section~\ref{subs:optim}. For any 
$\gamma\in \Delta(\h_j)\setminus \{\beta_j\}=:\ov{\Delta(\h_j)}$, 
we have $\beta_j-\gamma\in\Delta^+$. 

\begin{lm}                 \label{lm:roots-of-stab}
Under the previous notation, if\/ $\n\ne \tilde\n$ and 
$\gamma\in \Delta(\h_j)\setminus \Delta(\n)=:\eus C_\n(j)\subset\ov{\Delta(\h_j)}$, then the root space
$\g_{\beta_j-\gamma}$ belongs to $\n^\bxi$.
\end{lm}
\begin{proof}
Assume that $\beta_j-\gamma\not \in \Delta(\n)$. Then $[\gamma:\ap]=[\beta_j-\gamma:\ap]=0$ for
all $\ap\in {\eus T}=\Delta(\n)\cap\Pi$. Hence $[\beta_j:\ap]=0$. However, since $\beta_j\in\Delta(\n)$, 
there is an $\ap'\in {\eus T}(\n)$ such that $\beta_j\succ \ap'$, i.e., $[\beta_j:\ap']>0$. This contradiction 
shows that $\gamma':=\beta_j-\gamma\in \Delta(\n)$. 

Let us prove that $\g_{\gamma'}\subset \n^\bxi$. Since $\gamma'\in \Delta(\h_j)$ and hence
$\supp(\gamma')\subset\supp(\beta_j)$, properties of $\eus K$ imply that $\beta_j$ is the only 
element $\beta_i$ of $\eus K$ such that $\beta_i-\gamma'$ is a root. Indeed, 
\begin{itemize}
\item if $\beta_i$ and $\beta_j$ are incomparable, then $\supp(\beta_i)\cap\supp(\beta_j)=\varnothing$;
\item if $\beta_i\prec\beta_j$, then $\supp(\beta_i)\subset\supp(\beta_j)\setminus\Phi(\beta_j)$, while
$\supp(\gamma')\cap\Phi(\beta_j)\ne \varnothing$;
\item if $\beta_i\succ\beta_j$, then $(\beta_i,\gamma')=0$ and $\gamma'$ does not belong to the 
Heisenberg subset of $\Delta\lg i\rg$.
\end{itemize}
Therefore, $\gamma'-\beta_j=-\gamma$ is the only root that might 
occur in $[e_{\gamma'},\bxi]$. But $\gamma\not\in \Delta(\n)$, i.e.,
$-\gamma\not\in \Delta(\n^*)$ and therefore $\ads_\n(e_{\gamma'})(\bxi)=0$.
\end{proof}

\begin{thm}   \label{thm:cascade-point-stab}
Let $\n$ be an arbitrary nilradical and $\bxi\in\ce_0\subset\n^*$ a cascade point. Then
\begin{itemize}
\item[\sf (i)] \  $\n^\bxi$ is $T$-stable and
\[ 
   \n^\bxi=\Bigl(\bigoplus_{j=1}^k \bigoplus_{\gamma\in \eus C_\n(j)} \g_{\beta_j-\gamma}\Bigr)\oplus 
   \tilde\n^{\tilde\bxi}= \Bigl(\bigoplus_{j=1}^k \bigoplus_{\gamma\in \eus C_\n(j)} \g_{\beta_j-\gamma}\Bigr)
\oplus( \bigoplus_{\beta\in\eus K(\n)}\g_\beta) .
\] 
\item[\sf (ii)] \  
$\Delta(\n^\bxi)=\{\beta_j-\gamma\mid  1\le j\le k\ \ \&\ \ \gamma\in\eus C_\n(j)\}\bigsqcup \eus K(\n)$;
\item[\sf (iii)] \  For any $\xi\in \ce_0$, we have $\n^\xi=\n^\bxi$.
\end{itemize}
\end{thm}
\begin{proof}
The number of roots in $\bigsqcup_{j=1}^k \eus C_\n(j)$ equals 
$\dim(\tilde\n/\n)=\dim\n^\bxi-\dim\tilde\n^{\tilde\bxi}$, and each such root yields a root subspace in 
$\n^\bxi$ (Lemma~\ref{lm:roots-of-stab}), hence {\sf (i)}. Clearly, {\sf (ii)} is just a reformulation of {\sf (i)}.
The last assertion follows from the fact that $\n^\bxi$ is $T$-stable and $T{\cdot}{\bxi}=\ce_0$.
\end{proof}
The advantage of $\bxi\in\ce_0$ is that $\n^\bxi$ is a sum of root spaces. Therefore, 
Eq.~\eqref{eq:krit-sgp} is easily verified in practice. It is convenient to restate it as follows.

\begin{utv}   \label{utv:restate}
The stabiliser $\n^\bxi$ of $\bxi\in\ce_0\subset \n^*$ is generic (equivalently, $(\n:\n^*)$ has a 
generic stabiliser) if and only if
the difference of any two roots in $\Delta(\n^\bxi)$ does not belong to $\Delta(\n)$.
\end{utv}

\begin{ex}      \label{ex:sl7_2-6}
(1) If $\g=\mathfrak{sl}_7$ and ${\eus T}_1=\{\ap_2,\ap_6\}$, then $\n=\n_{\eus T_1}$ is not optimal and 
$\tilde\n=\n_{\tilde{\eus T}}$, where $\tilde{\eus T}=\{\ap_1,\ap_2,\ap_5,\ap_6\}$, see Example~\ref{ex:sln}.
Here $\eus K(\n)=\{\beta_1,\beta_2\}=\{\ap_1+\dots+\ap_6,\ap_2+\dots+\ap_5\}$ 
and the matrices therein show that
$\dim(\tilde\n/\n)=4$. More precisely, 
\begin{gather*}
\eus C_\n(1)=\Delta(\h_1)\setminus \Delta(\n\cap\h_1)=\{\ap_1\},  \\ 
\eus C_\n(2)=\Delta(\h_2)\setminus \Delta(\n\cap\h_2)=\{\ap_5, \ap_4+\ap_5, \ap_3+\ap_4+\ap_5={:}[3,5]\} .
\end{gather*}
Therefore, $\Delta(\n^\bxi)=\{
\ap_2, \ap_2+\ap_3, [2,4], [2,5], [2,6], [1,6]\}$. 
Since $[2,6]-\ap_2=[3,6]\in \Delta(\n)$, Condition~\ref{utv:restate} is not satisfied for $\bxi$, 
and $(\n:\n^*)$ does not have a generic stabiliser.

(2) If $\g=\mathfrak{sl}_7$ and ${\eus T}_2=\{\ap_1,\ap_2,\ap_6\}$, then $\n_{T_2}$ has the same optimisation $\tilde\n$ as $\n_{\eus T_1}$, but now $\dim(\tilde\n/\n_{\eus T_2})=3$ and
$\Delta(\n_{\eus T_2}^\bxi)=\{ \ap_2, \ap_2+\ap_3, [2,4], [2,5], [1,6]\}$. Here Condition~\ref{utv:restate} is satisfied and $\n_{\eus T_2}^\bxi$ is 
a generic stabiliser.
\end{ex}

Thus, the action $(N:\n^*)$ does not always have $N$-generic stabilisers.  A remedy
is to consider the action of the larger group $\tilde N$ on $\n^*$.

\begin{thm}   \label{thm:dense-saturation}
For any nilradical $\n$ and the cascade subspace $\ce\subset\n^*$, we have
\begin{itemize}
\item[\sf(i)] \ $\dim N{\cdot}\ce=2\dim\n-\dim\tilde\n$, i.e., $\codim_{\n^*}N{\cdot}\ce=\dim(\tilde\n/\n)$;
\item[\sf(ii)] \ the $\tilde N$-saturation of\/ $\ce$ is dense in $\n^*$, i.e., $\ov{\tilde N{\cdot}\ce}=\n^*$; 
\item[\sf(iii)] \ $\tilde\n^\bxi=\n^\bxi$ is a generic stabiliser for the linear action $(\tilde N:\n^*)$.
\end{itemize}
\end{thm}
\begin{proof}
Recall that $\dim\ce=\ind\tilde\n$ and \ 
$\max_{\xi\in\n^*}\dim N{\cdot}\xi=\dim N{\cdot}\bxi=\dim\n-\ind\n$. 

{\sf (i)} \ Since $\mathsf T_\bxi(N{\cdot}\bxi)=\n{\cdot}\bxi=(\n^\bxi)^\perp$, it follows from 
Theorem~\ref{thm:cascade-point-stab}{\sf (iii)} that all tangent spaces $\n{\cdot}\xi$, $\xi\in\ce_0$, are 
the same. Therefore, using the differential of  the map 
\[
   \kappa: N\times\ce\to \ov{N{\cdot}\ce}\subset\n^*
\]
at $(1,\bxi)$, we obtain $\dim\ov{N{\cdot}\ce}=\dim(\Ima \textsl{d}\kappa_{(1,\bxi)})=
\dim(\n{\cdot}\bxi+\ce)$. Because the roots in $\eus K$ are strongly orthogonal,  we have 
$\n{\cdot}\bxi\cap\ce=\{0\}$ and hence
\[
 \dim(\n{\cdot}\bxi+\ce)=\dim\n-\ind\n+\ind\tilde\n=2\dim\n-\dim\tilde\n .
\]

{\sf (ii)} \ Since  $\tilde N{\cdot}\tilde\ce$ is dense in $\tilde\n^*$ (Proposition~\ref{prop:optim-sgp}) and 
$\tau$ is $B$-equivariant, we conclude that 
$\tau(\tilde N{\cdot}\tilde\ce)=\tilde N{\cdot}\ce$ is dense in $\tau(\tilde\n^*)=\n^*$. 

{\sf (iii)} \ Let us first prove that
\beq      \label{eq:max-max}
    \max_{\xi\in\n^*}\dim \tilde N{\cdot}\xi=\max_{\xi\in\n^*}\dim N{\cdot}\xi+\dim(\tilde\n/\n) .
\eeq
Clearly,  inequality ``$\le$'' holds. By Eq.~\eqref{eq:compare-ind}, the RHS equals
$\dim\tilde\n-\ind\n=\dim\n-\dim\ce$. On the other hand, part~{\sf (ii)} implies that 
$\dim\n=\dim \tilde N{\cdot}\ce\le \max_{\xi\in\ce}\dim \tilde N{\cdot}\xi+\dim\ce$. This 
proves~Eq.~\eqref{eq:max-max} and also shows that almost all $\xi\in \ce$ satisfy this relation.
Moreover, since $\ce$ is $T$-stable and $\ce_0=T{\cdot}\bxi$ is dense in $\ce$, we obtain that 
\[
          \dim \tilde N{\cdot}\xi=\dim N{\cdot}\xi  + \dim(\tilde\n/\n)
\]
for every $\xi\in \ce_0$.  Hence $\tilde\n^\xi=\n^\xi$ for every $\xi\in \ce_0$. Combining this with 
Theorem~\ref{thm:cascade-point-stab}{\sf (iii)} and part~{\sf (ii)}, we conclude that 
$\tilde\n^\bxi=\n^\bxi$ is a generic stabiliser for the action $(\tilde N:\n^*)$.
\end{proof}

Using Theorem~\ref{thm:cascade-point-stab}, 
we describe certain nilradicals that do not have a generic stabiliser. 

\begin{prop}    \label{prop:ohne-s.g.p.}
Let $\beta_j$ be a descendant of $\beta_i\in\eus K$ and $\ap\in\Phi(\beta_i)$.
Suppose that  
$\ap\not\in \eus T$ and $[\beta_i:\nu]>[\beta_j:\nu]>0$ for some $\nu\in \eus T$. Then $\n=\n_{\eus T}$ 
does not have a generic stabiliser.
\end{prop}
\begin{proof}
Recall that $\beta_i$ is the highest root in $\Delta\lg i\rg^+$
and $\beta_j$ is a maximal root in $\Delta\lg i\rg^+\setminus \Delta(\h_i)$. If $\Phi(\beta_i)=\{\ap\}$, then 
$\ap+\beta_j\in \Delta(\h_i)\subset \Delta\lg i\rg^+$, because $\beta_j$ is not the highest root of 
$\Delta\lg i\rg^+$. If $\Phi(\beta_i)=\{\ap,\ap'\}$, then $\Delta\lg i\rg$ is of type $\GR{A}{p}$ and hence 
both $\beta_j+\ap$ and $\beta_j+\ap'$ belong to $\Delta(\h_i)$. In any case, if $\ap\in \Phi(\beta_i)$, then 
$\beta_i-\ap-\beta_j\in \Delta^+$.

Since $\nu\in \eus T$ and $[\beta_j:\nu]>0$, we have $\beta_j\in\Delta(\n)$ and also $\beta_i\in\Delta(\n)$.
That is, $\beta_i,\beta_j\in \eus K(\n)$.
Because $\ap\in \Delta(\h_i)\setminus \Delta(\n)$, we have $\beta_i-\ap\in\Delta(\n^\bxi)$, see 
Theorem~\ref{thm:cascade-point-stab}{\sf (ii)}. The assumption on $\nu$ implies that 
$[\beta_i-\ap-\beta_j:\nu]>0$. Thus, we have $\beta_i-\ap,\beta_j\in\Delta(\n^\bxi)$ and 
$\beta_i-\ap-\beta_j\in\Delta(\n)$. By Condition~\ref{utv:restate}, this means that $\n$ has no generic stabilisers.
\end{proof}

\begin{ex}    \label{ex:ohne-sgp}
For applications, it suffices to consider the case in which $\beta_i=\theta$, i.e., $i=1$. 

(1) \ If $\g$ is exceptional, then $\#\Phi(\beta_1)=1$ and $\beta_1$ has the unique descendant 
$\beta_2$. Although $\beta_1-\beta_2$ is not a root, its support is defined as in
Section~\ref{subs:root-order}. Then Proposition~\ref{prop:ohne-s.g.p.} applies to any 
$\nu\in \supp(\beta_1-\beta_2)\setminus \Phi(\beta_1)$ and ${\eus T}=\{\nu\}$. 
\\ \indent
For instance, let $\g$ be of type $\GR{E}{6}$. Then $\Phi(\beta_1)=\{\ap_6\}$, 
$\supp(\beta_1-\beta_2)=\{\ap_2,\ap_3,\ap_4,\ap_6\}$, and
$\supp(\beta_2)\supset \{\ap_2,\ap_3,\ap_4\}$, see formulae for $\eus K$ in Appendix~\ref{sect:tables}. 
Hence one can take $\nu=\ap_i$ with $i\in\{2,3,4\}$. More generally, if
$\ap_6\not\in \eus T$ and ${\eus T}\cap \{\ap_2,\ap_3,\ap_4\}\ne\varnothing$, then $\n_\eus T$ has no generic stabilisers.
\\ \indent
(2) \ 
This method also works for $\g=\son$ with $n\ge 7$, where $\Phi(\beta_1)=\{\ap_2\}$. 
\\
-- \ If $n\ne 8$, then $\beta_1=\esi_1+\esi_2$ has two descendants, $\beta_2=\esi_1-\esi_2$ and 
$\beta_3=\esi_3+\esi_4$ ($\esi_4=0$, if $n=7$), see Appendix~\ref{sect:tables}. For $(\beta_1,\beta_3)$, 
Proposition~\ref{prop:ohne-s.g.p.} applies, if one takes ${\eus T}=\{\ap_3\}$. While for 
$(\beta_1,\beta_2)$, we can take ${\eus T}$ such that $\ap_2\not \in \eus T$ and 
${\eus T}\supset\{\ap_1,\ap_j\}$ with $j\ge 3$. 
\\
-- \ If $n=8$, then $\beta_1$ has three descendants, and we can take ${\eus T}=\{\ap_i,\ap_j\}$
with  $i,j\ne 2$.
\end{ex}

\begin{prop}    \label{prop:embedd-F(n)}
If\/ $\n\subset\n'$ and $\tilde\n=\tilde{\n'}$, then $\cF(\n')\subset \cF(\n)\subset \n$.
\end{prop}
\begin{proof}
Let $\ce$ (resp. $\ce'$) be the cascade subspace of $\n^*$ (resp. $(\n')^*$). Take 
$\bxi\in\ce_0$ and $\bxi'\in\ce'_0$. Since $\Delta(\n)\subset\Delta(\n')$ and $\eus K(\n)=\eus K(\n')$,
we have $\eus C_\n(j)\supset \eus C_{\n'}(j)$ for all $\beta_j\in \eus K(\n)$. Then 
Theorem~\ref{thm:cascade-point-stab} shows that $\n^\bxi\supset (\n')^{\bxi'}$. By Proposition~\ref{thm:F-&-sgp},
this yields the required embedding.
\end{proof}

Our description of $\Delta(\n^\bxi)$ yields a criterion for $\n$ to be quasi-quadratic.

\begin{thm}      \label{thm:F(n)=n}
For a nilradical $\n=\n_{\eus T}$, the following assertions are equivalent:
\begin{itemize}
\item \ $\cF(\n)=\n$, i.e., $\n$ is quasi-quadratic;
\item \ for each $\ap\in \eus T$, one of the two conditions is satisfied:
\begin{itemize}
\item either $\ap\in\eus K$;
\item or $\Phi(\Phi^{-1}(\ap))=\{\ap,\ap'\}$, and if $C$ is the chain in the Dynkin diagram that connects
$\ap$ and $\ap'$, then $C\cap {\eus T}=\{\ap\}$ (in particular, $\ap'\not\in \eus T$).
\end{itemize}
\end{itemize}
\end{thm}
\begin{proof}
 Since $\cF(\n)$ is the $\be$-ideal generated by $\n^\bxi$ (Prop.~\ref{thm:F-&-sgp}), it is clear that
 $\cF(\n)=\n$ if and only if $\ap\in\Delta(\n^\bxi)$ for each $\ap\in \eus T$.
 
(1) \ If $\ap\in \eus T\cap\eus K$, then $\Phi(\ap)=\ap$ and $\ap\in \Delta(\tilde\n^\bxi)\subset \Delta(\n^\bxi)$.
 
(2) \ If $\ap\in \eus T\setminus \eus K$ and $\Phi^{-1}(\ap)=\beta_j\in \eus K_{\eus T}$, then 
$\ap\ne\beta_j$ and there are two possibilities.
 
{\bf --} \ $\Phi(\beta_j)=\ap$. Then the whole Heisenberg algebra $\h_{j}$ belongs to $\n$ and hence 
$\eus C_\n(j)=\varnothing$. Then it follows from Theorem~\ref{thm:cascade-point-stab}{\sf (ii)} that 
$\Delta(\h_j\cap \n^\bxi)=\{\beta_j\}$, i.e., $\ap\not\in \Delta(\n^\bxi)$.

{\bf --} \ $\Phi(\beta_j)=\{\ap, \ap'\}$. Here $\beta_j$ is the highest root in a root system of type 
$\GR{A}{p}$ ($p\ge 2$). Therefore, if $C=\{\ap=\ap_1,\ap_2,\dots,\ap_p=\ap'\}$ is the chain connecting 
$\ap$ and $\ap'$, then $\beta_j=\ap_1+\ap_2+\dots +\ap_p$. Then $\ap\in\Delta(\n^\bxi)$ if and only if 
$\beta_j-\ap=\ap_2+\dots+\ap_p\not\in\Delta(\n)$, and this is only possible if 
$\supp(\beta_j)\cap {\eus T}=C\cap {\eus T}=\{\ap\}$.
\end{proof}

Using Theorem~\ref{thm:F(n)=n}, one readily obtains the list of all quasi-quadratic nilradicals.

\begin{prop}  \label{list:quasi-quadr}  
The quasi-quadratic nilradicals are as follows.
\begin{enumerate}
\item \  If\/ $\g$ is not of type $\GR{A}{n},\GR{D}{2n+1},\text{ and } \GR{E}{6}$, then
 $\cF(\n_{\eus T})=\n_{\eus T}$ if and only if\/ ${\eus T}\subset \eus K$;
  
\item \ If\/ $\g$ is of type $\GR{A}{n}$, then $\cF(\n_{\eus T})=\n_{\eus T}$ if and only if\/ ${\eus T}=\{\ap_i\}$, $i=1,2,\dots,n$;

\item \ If\/ $\g$ is of type $\GR{D}{2n+1}$, then $\cF(\n_{\eus T})=\n_{\eus T}$ if and only if\/ 
${\eus T}\cap\{\ap_2,\ap_4,\dots,\ap_{2n-2}\}=\varnothing$ and 
$\#({\eus T}\cap\{\ap_{2n-1},\ap_{2n},\ap_{2n+1}\})\le1$;

\item \ If\/ $\g$ is of type\/ $\GR{E}{6}$, then $\cF(\n_{\eus T})=\n_{\eus T}$ if and only if\/ 
${\eus T}=\{\ap_i\}$, $i\ne 6$.
\end{enumerate}
\end{prop}
\begin{proof}
(1) In this case $\Phi^{-1}$ is a bijection, hence $\Phi(\Phi^{-1}(\ap))=\{\ap\}$ for any $\ap\in\Pi$.

(2) Here each $\n$ has a \CP~\cite{ag03}, hence $\cF(\n)$ is abelian. That is, $\cF(\n)=\n$ must be an abelian nilradical.

(3) In this case, $\Pi\cap\eus K=\{\ap_1,\ap_3,\dots,\ap_{2n-1}\}$ and there is a unique $\beta\in\eus K$
such that $\#\Phi(\beta)=2$. Namely, $\Phi(\beta_{2n-1})=\{\ap_{2n},\ap_{2n+1}\}$, see 
Appendix~\ref{sect:tables}. The chain connecting $\ap_{2n}$ and $\ap_{2n+1}$ in the 
Dynkin diagram is $C=\{\ap_{2n},\ap_{2n-1}, \ap_{2n+1}\}$. Hence the answer.

(4)  Here $\Pi\cap\eus K=\{\ap_3\}=\{\beta_4\}$, $\Phi(\beta_3)=\{\ap_2,\ap_4\}$, $\Phi(\beta_2)=\{\ap_1,\ap_5\}$, see Appendix~\ref{sect:tables}. Since $\ap_1,\ap_2,\ap_3,\ap_4,\ap_5$ form a chain in the Dynkin diagram, the result follows.
\end{proof}

\section{Generic stabilisers and the Frobenius semiradical in case of $\slno$ or $\spn$}
\label{sect:generic-sl-sp}

\noindent
First, we explicitly describe the nilradicals in $\slno$ or $\spn$ having a generic stabiliser. To state the 
result for $\slno$, we need some notation.  Recall that, for $\slno$,  the poset $\eus K$ is a chain $\beta_1\succ\beta_2\succ\dots\succ\beta_t$, where 
$t=[(n+1)/2]$ and $\Phi(\beta_i)=\{\ap_i,\ap_{n+1-i}\}$. Let $\n=\n_{\eus T}\subset\slno$ be a nilradical. Then $\eus K(\n)=\{\beta_1,\dots,\beta_k\}$ for some $k\le t$, cf. Lemma~\ref{lm:optim}. 
Therefore, ${\eus T}\subset \{\ap_1,\dots,\ap_k,\ap_{n+1-k},\dots, \ap_n\}$ and 
${\eus T}\cap\{\ap_k,\ap_{n+1-k}\}\ne\varnothing$. Set ${\eus T}'=\eus T\cap \{\ap_1,\dots,\ap_{k-1}\}$ and
${\eus T}''=\eus T\cap\{\ap_{n+2-k},\dots, \ap_n\}$.

\begin{thm}    \label{thm:sop_sln}
Suppose that\/ $\g=\slno$ and $\eus K(\n)=\{\beta_1,\dots,\beta_k\}$. Then 
\[
  \text{$\n^\bxi$ is a generic stabiliser $\Leftrightarrow$ $\sigma(\eus T')=\eus T''$, where $\sigma$ is 
  the symmetry of Dynkin diagram $\GR{A}{n}$.} 
\] 
(That is, $\sigma(\ap_j)=\ap_{n+1-j}$.)
In particular, if $\sigma(\eus T)=\eus T$, then $\n_\eus T$ has a generic stabiliser.
\end{thm}
\begin{proof}
{\sf 1$^o$.}  For $k=1$, we have $\eus T'=\eus T''=\varnothing$ and $\n\subset \h_1$. If $\eus T=\{\ap_1,\ap_n\}$,
then $\n=\h_1$ is optimal. If $\eus T=\{\ap_1\}$, then $\n$ is abelian. In both cases, there is a generic 
stabiliser for $(\n:\n^*)$, as required.

{\sf 2$^o$.} Suppose that $k\ge 2$ and the symmetry of ${\eus T}'$ and ${\eus T}''$ fails for some $j\le k-1$. W.l.o.g, we may assume that 
$\ap_j\in {\eus T}'$, whereas $\ap_{n+1-j}\not\in {\eus T}''$.
Then $\beta_j-\ap_{n+1-j}=\beta_{j+1}+\ap_j\in \Delta(\n^\bxi)$. Since $\beta_{j+1}\in \eus K(\n)\subset
\Delta(\n^\bxi)$ and $\ap_j\in \Delta(\n)$,  Condition~\ref{utv:restate} is not satisfied.
 
{\sf 3$^o$.} If $\sigma (\eus T')=\eus T''$, then one can directly describe $\n^\bxi$ and see 
that~\ref{utv:restate} is satisfied. It is necessary to distinguish two cases:
{\bf (a)} $\ap_k, \ap_{n+1-k}\in {\eus T}$, {\bf (b)} only $\ap_k\in {\eus T}$.

{\bf (a)} Here $\sigma(\eus T)=\eus T$ and $\Phi(\beta_k)\subset\Delta(\n)$. Hence $\h_k\subset\n$ and 
$\eus C_\n(k)=\varnothing$.  
Theorem~\ref{thm:cascade-point-stab}{\sf (ii)} shows that
$\n^\bxi\subset \n_{\{\ap_k\}}\cap\n_{\{\ap_{n+1-k}\}}$, the last intersection being  
the north-east $k\times k$ square of $(n+1)\times (n+1)$ matrices in $\slno$. More precisely, if
${\eus T}'\cup\{\ap_k\}=\{\ap_{i_1}, \ap_{i_2},\dots,\ap_{i_j}=\ap_k\}$, then 
$\Delta(\n^\bxi)=\Gamma_1\cup\dots \cup \Gamma_j$, where
\begin{gather*}
\Gamma_1=\{\esi_i-\esi_j \mid 1\le i\le i_1, \ n+2-i_1\le j\le n+1\} , \\ 
\Gamma_2=\{\esi_i-\esi_j \mid i_1+1\le i\le i_2, \ n+2-i_2\le j\le n+1-i_1\}, \text{etc.}
\end{gather*}
Here $\{\Gamma_s\}_{s=1}^j$ is a string of square blocks located along the antidiagonal in the 
$k\times k$ square, see Figure~\ref{fig:cases-a&b}(a). The size of the $s$-th square is $i_s-i_{s-1}$, 
where $i_0=0$. It is readily seen that if $\gamma$ and $\gamma'$ belong to different blocks, then 
$\gamma-\gamma'$ is not a root; while if $\gamma$ and $\gamma'$ belong to the same block, then 
$\gamma-\gamma'$ is either not a root or a root of the standard Levi subalgebra corresponding to $\n$. 
Thus, $\gamma-\gamma'\not\in\Delta(\n)$ and~\ref{utv:restate} is satisfied.

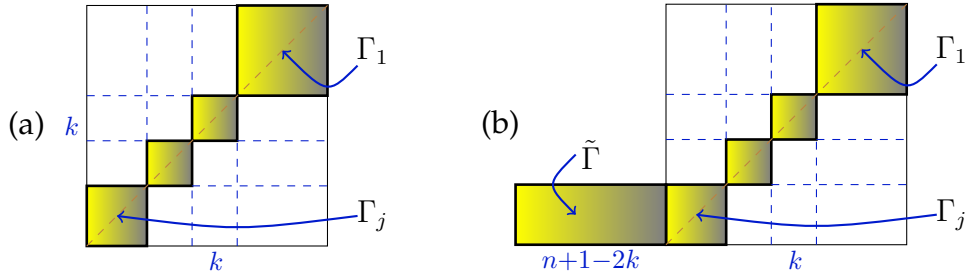
\begin{figure}[htb]    
\label{fig:a&b}
\begin{center}  
\begin{tikzpicture}[scale= .40]
\draw (0,0)  rectangle (8,8);

\shade[left color=yellow,right color=gray, draw,line width=1pt] (5,5) -- (5,8) -- (8,8) -- (8,5)--cycle ;
\shade[left color=yellow,right color=gray, draw,line width=1pt] (3.5,3.5) -- (3.5,5) -- (5,5) -- (5,3.5)--cycle ;
\shade[left color=yellow,right color=gray, draw,line width=1pt] (2,2) -- (2,3.5) -- (3.5,3.5) -- (3.5,2)--cycle ;
\shade[left color=yellow,right color=gray, draw,line width=1pt] (0,0) -- (0,2) -- (2,2) -- (2,0)--cycle ;

\draw[dashed,darkblue]  (5,0) -- (5,3.5) ;
\draw[dashed,darkblue]  (3.5,0) -- (3.5,2) ;
\draw[dashed,darkblue]  (3.5,5) -- (3.5,8) ;
\draw[dashed,darkblue]  (2,3.5) -- (2,8) ;

\draw[dashed,darkblue]  (0,5) -- (3.5,5) ;
\draw[dashed,darkblue]  (0,3.5) -- (2,3.5) ;
\draw[dashed,darkblue]  (5,3.5) -- (8,3.5) ;
\draw[dashed,darkblue]  (3.5,2) -- (8,2) ;

\draw (9.5,6.5)  node {$\Gamma_1$} ;
\draw[->,thick,darkblue]   (9,6) .. controls (8,5) .. (6.5,6.5);

\draw (9.5,1)  node {$\Gamma_j$} ;
\draw[->,thick,darkblue]   (9,1) .. controls (5,.5) .. (1,1);

\draw (-2,4)  node {(a)} ;
\draw (-0.5,4)  node {\footnotesize {\color{darkblue}$k$}} ;
\draw (4.3,-0.5)  node {\footnotesize {\color{darkblue}$k$}} ;
\draw[dashed,brown]  (0,0) -- (8,8) ;
\end{tikzpicture}
\qquad
\begin{tikzpicture}[scale= .40]
\draw (0,0)  rectangle (8,8);

\shade[left color=yellow,right color=gray, draw,line width=1pt] (5,5) -- (5,8) -- (8,8) -- (8,5)--cycle ;
\shade[left color=yellow,right color=gray, draw,line width=1pt] (3.5,3.5) -- (3.5,5) -- (5,5) -- (5,3.5)--cycle ;
\shade[left color=yellow,right color=gray, draw,line width=1pt] (2,2) -- (2,3.5) -- (3.5,3.5) -- (3.5,2)--cycle ;
\shade[left color=yellow,right color=gray, draw,line width=1pt] (0,0) -- (0,2) -- (2,2) -- (2,0)--cycle ;

\shade[left color=yellow,right color=gray, draw,line width=1pt] (-5,0) -- (0,0) -- (0,2) -- (-5,2)--cycle ;

\draw[dashed,darkblue]  (5,0) -- (5,3.5) ;
\draw[dashed,darkblue]  (3.5,0) -- (3.5,2) ;
\draw[dashed,darkblue]  (3.5,5) -- (3.5,8) ;
\draw[dashed,darkblue]  (2,3.5) -- (2,8) ;

\draw[dashed,darkblue]  (0,5) -- (3.5,5) ;
\draw[dashed,darkblue]  (0,3.5) -- (2,3.5) ;
\draw[dashed,darkblue]  (5,3.5) -- (8,3.5) ;
\draw[dashed,darkblue]  (3.5,2) -- (8,2) ;

\draw (9.5,6.5)  node {$\Gamma_1$} ;
\draw[->,thick,darkblue]   (9,6) .. controls (8,5) .. (6.5,6.5);

\draw (9.5,1)  node {$\Gamma_j$} ;
\draw[->,thick,darkblue]   (9,1) .. controls (5,.5) .. (1,1);

\draw (-2.5,3)  node {$\tilde\Gamma$} ;
\draw[->,thick,darkblue]   (-3,3) .. controls (-4,2) .. (-3,1);

\draw (4.3,-0.5)  node {\footnotesize {\color{darkblue}$k$}} ;
\draw (-2.5,-0.5)  node {\footnotesize {\color{darkblue}$n{+}1{-}2k$}} ;
\draw[dashed,brown]  (0,0) -- (8,8) ;
\draw (-5.5,4)  node {(b)} ;
\end{tikzpicture}
\caption{The generic stabiliser $\n^\bxi$: cases (a) and (b)}
\label{fig:cases-a&b}
\end{center}
\end{figure}

{\bf (b)} Here $\n$ is smaller than in part {\bf (a)}, but $\eus K(\n)$ remains the same. Now only 
``half'' of $\h_k$ belongs to $\n$. Therefore, $\eus C_\n(k)\ne \varnothing$, and the subsets 
$\eus C_\n(s)$ with $i_{j-1}< s< k$ become larger than in {\bf (a)}. Hence $\Delta(\n^\bxi)$ becomes larger and, along with $\Gamma_1\cup\dots \cup \Gamma_j$, it also contains the strip of roots  
\\[.6ex]
\centerline{$\tilde\Gamma=\{\esi_i-\esi_j\mid i_{j-1}+1\le i\le i_j=k, \ k+1\le j\le n+1-k\}$}
\\[.6ex] 
attached to $\Gamma_j$, see Figure~\ref{fig:cases-a&b}(b). 
The new feature is that the difference of roots from $\Gamma_j$ and $\tilde\Gamma$ can be a root in 
$\bigcup_{s=i_{j-1}}^k \Delta(\h_s)$. But such a difference belongs to the parts of sets $\Delta(\h_s)$ 
that are missing in $\Delta(\n)$. Thus,~\ref{utv:restate} is still satisfied here.
\end{proof}

The symmetry condition of Theorem~\ref{thm:sop_sln} means that the matrix shape of 
$\n\subset\slno$ must be `almost' symmetric w.r.t. the antidiagonal. That is, the symmetry may only fail 
in case {\bf (b)} for roots in $\Delta(\h_s)$ with $i_{j-1}< s \le i_j=k$ (if 
$\ap_k\in {\eus T}$, but $\ap_{n+1-k}\not\in {\eus T}$).
\begin{rmk}   \label{rem:podschet-An}
Using Theorem~\ref{thm:sop_sln}, one easily computes the number of nontrivial (standard) nilradicals 
with generic stabilisers. For $\GR{A}{2n-1}$, it is $2^{n+1}-3$; 
for $\GR{A}{2n}$, it is $3(2^{n}-1)$. Hence the ratio 
$\#\{\text{nilradicals with generic stabiliser}\}/\#\{\text {all nilradicals}\}$ exponentially decreases.
\end{rmk}
\begin{thm}    
\label{thm:sop_spn}
If\/ $\g=\spn$, then a generic stabiliser exists for every nilradical\/ $\n$.  
\end{thm}
\begin{proof}  For an appropriate choice of a skew-symmetric bilinear form defining $\spn\subset\sltn$,
a Borel subalgebra of $\spn$, $\be(\spn)$, is the set of symplectic upper-triangular matrices, see
Example~\ref{ex:spn}. Then the matrix shape of {\bf any} (standard) nilradical in $\spn$ is symmetric w.r.t. 
the antidiagonal, and the approach in the proof of 3$^o${\bf (a)} in Theorem~\ref{thm:sop_sln} applies 
to any nilradical in $\spn$. If $\n=\n_{\eus T}$ and ${\eus T}=\{\ap_{i_1},\dots,\ap_{i_j}\}$, where
$i_1<i_2<\dots < i_j=k$, then $\n^\bxi$ belongs to the north-east $k\times k$ square in the abelian 
nilradical $\napn$. Here the blocks $\Gamma_j$ inside this square represent matrices that are symmetric 
w.r.t. the antidiagonal, cf. Figure~\ref{fig:cases-a&b}(a). 
\end{proof}

For these two series, we explicitly describe $\cF(\n)$ for any $\n$. This also demonstrates the role of commutative polarisations.

\subsection{The Frobenius semiradical $\cF(\n)$ for $\g=\slno$}
\label{subs:frob-semi-sl}

Let $\n=\n_{\eus T}$ be a nilradical such that $\min\eus K(\n)=\{\beta_k\}$. Then 
$\eus K(\n)=\{\beta_1,\dots,\beta_k\}$, ${\eus T}\subset\{\ap_1,\dots,\ap_k,\ap_{n+1-k},\dots,\ap_n\}$, and 
$\Phi(\beta_k)\cap {\eus T}\ne\varnothing$. We may assume that $\ap_k\in {\eus T}$ and then, as in the proof of Theorem~\ref{thm:sop_sln}, there are two possibilities. 

{\bf (a)} $\ap_{n+1-k}\in {\eus T}$. Then $\n_{\{\ap_k\}}$ and $\n_{\{\ap_{n+1-k}\}}$ are \CP-ideals of 
$\n$~(\cite[Example\,3.8]{CP}), and hence $\cF(\n)\subset \n_{\{\ap_k\}}\cap\n_{\{\ap_{n+1-k}\}}$, cf.
Section~\ref{subs:role}(3). Here
$\es:=\n_{\{\ap_k\}}\cap\n_{\{\ap_{n+1-k}\}}$ is the north-east square of size $k$ in $\slno$. On the other
hand, $\n^\bxi\subset \es$ (see Theorem~\ref{thm:sop_sln}) and $\be{\cdot}\n^\bxi=\es$.
Hence $\cF(\n)=\es$ and $\dim\cF(\n)=k^2$, see Figure~\ref{fig:c}. For the special case of $\n=\ut$, the
description of $\cF(\ut)$ is obtained in \cite[Theorem\,4.1]{ooms2}.

\begin{figure}[htb]  
\begin{center}  
\begin{tikzpicture}[scale= .4] 
\draw (0,0)  rectangle (6.5,6.5);

\shade[left color=yellow,right color=gray, draw,line width=1pt] (0,0) -- (0,6.5) -- (6.5,6.5) -- (6.5,0)--cycle ;

\draw[dashed,darkblue]  (5,-5) -- (5,6.5) ;           
\draw[dashed,darkblue]  (3.5,-5) -- (3.5,6.5) ;
\draw[dashed,darkblue]  (1.5,-5) -- (1.5,6.5) ;
\draw[dashed,darkblue]  (-1.5,0) -- (-1.5,6.5) ;
\draw[dashed,darkblue]  (-3.5,0) -- (-3.5,6.5) ;

\draw[dashed,darkblue]  (-5,5) -- (6.5,5) ;             
\draw[dashed,darkblue]  (-5,3.5) -- (6.5,3.5) ;
\draw[dashed,darkblue]  (-5,1.5) -- (6.5,1.5) ;
\draw[dashed,darkblue]  (0,-1.5) -- (6.5,-1.5) ;
\draw[dashed,darkblue]  (0,-3.5) -- (6.5,-3.5) ;

\draw (5.7,5.7)  node {\footnotesize {\color{darkblue}$\beta_1$}} ;
\draw (4.2,4.2)  node {\footnotesize {\color{darkblue}$\beta_2$}} ;
\draw (0.7,0.7)  node {\footnotesize {\color{darkblue}$\beta_k$}} ;
\draw (-4.3,0.7)  node {\footnotesize {\color{darkblue}$\ap_k$}} ;
\draw (-3,-4)  node {\footnotesize {\color{darkblue}$\ap_{n{+}1{-}k}$}} ;
\draw[->,thick,darkblue]   (-3,-3.5) .. controls (-1.1,-1) .. (.7,-4.3);

\draw (2.6,5.8)  node {\footnotesize {$\cdots$}} ;
\draw (2.6,4.05)  node {\footnotesize {$\cdots$}} ;
\draw (2.6,2.3)  node {\footnotesize {$\cdots$}} ;
\draw (2.6,.8)  node {\footnotesize {$\cdots$}} ;

\draw (2.6,-.7)  node {\footnotesize {$\cdots$}} ;
\draw (2.6,-2.45)  node {\footnotesize {$\cdots$}} ;
\draw (2.6,-4.2)  node {\footnotesize {$\cdots$}} ;
\draw (.9,-2.45)  node {\footnotesize {$\cdots$}} ;
\draw (4.3,-2.45)  node {\footnotesize {$\cdots$}} ;
\draw (5.8,-2.45)  node {\footnotesize {$\cdots$}} ;

\draw (-2.4,.8)  node {\footnotesize {$\cdots$}} ;
\draw (-2.4,2.3)  node {\footnotesize {$\cdots$}} ;
\draw (-2.4,4.05)  node {\footnotesize {$\cdots$}} ;
\draw (-2.4,5.8)  node {\footnotesize {$\cdots$}} ;

\draw (8.4,-2.5)  node {\footnotesize {\color{darkblue}$n{+}1{-}2k$}} ;
\draw (7.1,3)  node {\footnotesize {\color{darkblue}$k$}} ;
\draw (3.3,-5.6)  node {\footnotesize {\color{darkblue}$k$}} ;
\draw[dashed,brown]  (-3,-3) -- (6.5,6.5) ;
\draw[dashed,brown]  (-5.5,0.5) -- (0.5,-5.5) ;

\draw[thick,red]  (0,6.5) -- (-5,6.5) -- (-5,0) -- (6.5,0) ;
\draw[thick,forest]   (6.5,0) -- (6.5,-5) -- (0,-5) -- (0,6.5) ;

\end{tikzpicture}
\end{center}
\caption{$\cF(\n)=\n_{\{\ap_k\}}\cap \n_{\{\ap_{n+1-k}\}}$, \ 
$\n\supset\n_{\{\ap_k\}}\cup \n_{\{\ap_{n+1-k}\}}$}
\label{fig:c}
\end{figure}

Actually, here $\ah_j:=\n_{\{\ap_j\}}\cap \n$  is a \CP-ideal of $\n$ for any $j$ such that $k\le j\le n+1-k$, see Example~4.8 in \cite{CP}.

{\bf (b)} $\ap_{n+1-k}\not\in {\eus T}$. Then only $\n_{\{\ap_k\}}$ is a \CP-ideal and 
$\cF(\n)\subset \n_{\{\ap_k\}}$. On the other hand, 
$\supp(\beta_k-\ap_k)=\{\ap_{k+1},\dots,\ap_{n+1-k}\}$. Hence $\beta_k-\ap_k\not\in\Delta(\n)$ and 
$\ap_k\in\Delta(\n^\bxi)$, see Theorem~\ref{thm:cascade-point-stab}{\sf (ii)}. 
By Proposition~\ref{thm:F-&-sgp}{\sf (i)}, this implies 
that $\n_{\{\ap_k\}}\subset \cF(\n)$. Thus, here $\cF(\n)=\n_{\{\ap_k\}}$ and $\dim\cF(\n)=k(n+1-k)$.
Since $\cF(\n)$ appears to be a \CP, we conclude that $\n_{\{\ap_k\}}$ is the only \CP\ for $\n$.

It follows from the descriptions above that for $\g=\slno$, $\cF(\n)$ is always equal to the 
intersection of all \CP-ideals of $\n$. But this is not true for other simple Lie algebras.

\subsection{The Frobenius semiradical $\cF(\n)$ for $\g=\spn$}
\label{subs:frob-semi-sp}

Let $\n=\n_{\eus T}$ be a nilradical such that $\min\eus K(\n)=\{\beta_k\}$. Then 
$\eus K(\n)=\{\beta_1,\dots,\beta_k\}$, ${\eus T}\subset\{\ap_1,\dots,\ap_k\}$ and $\ap_k\in {\eus T}$. The abelian 
nilradical $\napn$ is identified with the space of $n\times n$ matrices that are symmetric w.r.t{.} the 
antidiagonal, see Example~\ref{ex:spn}. Then $\n^\bxi$ belongs to the north-east $k\times k$ 
square in $\napn$, and the south-west corner of this square, $\g_{\beta_k}$,  lies in $\n^\bxi$. 
Hence $\cF(\n)=\be{\cdot}\n^\bxi$ is equal to  this $k$-square
and $\dim \cF(\n)=k(k+1)/2$. In this case, $\ah:=\n_{\{\ap_n\}}\cap \n$  is the only \CP-ideal of $\n$
and the inclusion $\cF(\n)\subset \ah$ is proper unless $k=n$, see the following picture.

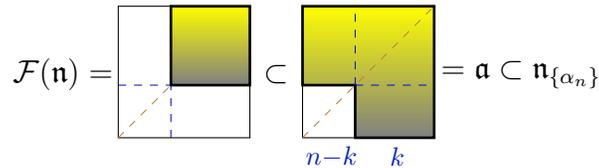
\begin{center}  
\begin{tikzpicture}[scale= .35] 
\draw (-7,0)  rectangle (-2,5);

\draw[dashed,brown]  (-7,0) -- (-2,5) ;
\draw[dashed,darkblue]  (-5,0) -- (-5,5) ;
\draw[dashed,darkblue]  (-7,2) -- (-2,2) ;

\shade[top color=yellow,bottom color=gray, draw, line width=1pt] (-5,2) -- (-2,2) -- (-2,5) -- (-5,5)--cycle ;
\draw (-1, 2.4) node {$\subset$};
\draw (-9.1, 2.4) node {$\cF(\n)=$};

\draw (0,0)  rectangle (5,5);

\shade[top color=yellow,bottom color=gray, draw, line width=1pt] (2,0) -- (5,0) -- (5,5) -- (0,5) -- (0,2) -- (2,2)--cycle ;
\draw[dashed,brown]  (0,0) -- (5,5) ;
\draw[dashed,darkblue]  (2,2) -- (2,5) ;
\draw[dashed,darkblue]  (2,2) -- (5,2) ;

\draw (1.1,-.7)  node {\footnotesize {\color{darkblue}$n{-}k$}} ;
\draw (3.6,-.7)  node {\footnotesize {\color{darkblue}$k$}} ;
\draw (8.3, 2.4) node {$=\ah\subset \napn$};
\end{tikzpicture}
\end{center}

\begin{rmk}     \label{rem:F(n)-abelian}
The series $\GR{A}{n}$ and $\GR{C}{n}$ are easily handled, because $\eus K$ is a chain for them and 
any nilradical has a \CP. 
Therefore, $\cF(\n)$ is an {\bf abelian} ideal of $\n$, and there is a natural upper bound on $\cF(\n)$. 
However, it can happen that $\n$ does not have a \CP, but $\cF(\n)$ is abelian, see 
Example~\ref{ex:metabelian}.
\end{rmk}

\section{The Poisson centre and $U$-invariants} 
\label{sect:centre-&-factor}

\noindent
Recall that $P=L{\cdot}N$ is a standard parabolic subgroup of $G$ and $\p=\el\oplus\n$. Since $N$ is 
connected, one has $\gS(\n)^N=\gS(\n)^\n$, and this algebra is the centre of the Poisson algebra 
$(\gS(\n),\{\ ,\, \})$. Because $\n$ is a $P$-module, one can consider algebras of invariants in $\gS(\n)$ 
for any subgroup $Q\subset P$. Specifically, we are interested in the groups $U$ and $\tilde N$, where 
$\Lie \tilde N=\tilde\n$ is the optimisation of $\n$. We also wish to compare the algebras of $Q$-invariants 
in $\gS(\n)$ and $\gS(\tilde\n)$. The algebras of interest for us are organised in the following diagram:
\beq   \label{eq:diagr}
    \begin{array}{ccccc}
    \gS(\tilde\n)^U   & \subset  & \gS(\tilde\n)^{\tilde N} & \\       
         \cup                    &               &   \cup         &              & \\
    \gS(\n)^U          & \subset  &\gS(\n)^{\tilde N}& \subset &\gS(\n)^N .
     \end{array}    
\eeq
If an algebra $\gS(\n)^Q=\BC[\n^*]^Q$ is finitely generated, then we also consider 
the associated quotient morphism $\pi_Q: \n^*\to \n^*\md Q:=\spe(\BC[\n^*]^Q)$.

For any nilradical $\n$, one can form the solvable Lie algebra
\[
       \ff_\n=\ff_{\tilde\n}=\te_\n\oplus\tilde\n \subset\be ,
\] 
where $\te_\n=\oplus_{\beta\in\eus K(\n)}[\g_\beta,\g_{-\beta}]\subset\te$. By~\cite[Prop.\,5.1]{CP},
the corresponding connected group $F_\n\subset B$ has an open orbit in $\ff_\n^*$, i.e., $\ff_n$ is a 
Frobenius Lie algebra. For this reason, $\ff_\n$ is called the {\it Frobenius envelope\/} of $\n$. Note that
the unipotent radical of $F_\n$ is $\tilde N$.

\begin{lm}    \label{lm:4.1}
The four algebras of invariants forming the square in \eqref{eq:diagr} are polynomial. In particular, for any 
{\bf optimal} nilradical $\tilde\n$, the Poisson centre $\gS(\tilde\n)^{\tilde N}$ is polynomial.
\end{lm}
\begin{proof}
 {\bf 1}$^o$. For any nilradical $\n$, the action $(B:\n^*)$ is locally transitive and $B{\cdot}\bxi$ is the dense 
orbit in $\n^*$~\cite[2.4]{jos77}. Therefore, $\gS(\n)^U$ is a polynomial algebra and
$\trdeg\gS(\n)^U$ equals the number of prime divisors in $\n^*\setminus B{\cdot}\bxi$, see~\cite[Theorem\,4.4]{p17}.
Hence, the algebras in the first column of \eqref{eq:diagr} are polynomial.

{\bf 2}$^o$. More generally, the $F_\n$-equivariant embeddings $\n\hookrightarrow \tilde\n\hookrightarrow \ff_\n$ yield
the $F_\n$-equivariant projections $\ff_\n^*\to \tilde\n^*\to  \n^*$, which implies that $F_\n$ has a dense 
orbit in $\n^*$, too. Therefore, arguing as in \cite[Section\,4]{p17}, one proves that $\gS(\n)^{\tilde N}$ is 
also a polynomial algebra. Thus, the algebras in the second column of \eqref{eq:diagr} are polynomial, 
too. 
\end{proof}

If $\n\ne\tilde\n$, then the Poisson centre $\gS(\n)^N$ is not always polynomial, see examples below.
\begin{lm}    \label{lm:4.2}
For any nilradical $\n$, we have $\gS(\n)^{\tilde N}=\gS(\tilde\n)^{\tilde N}$.
\end{lm}
\begin{proof}
It is clear that $\gS(\n)^{\tilde N}\subset\gS(\tilde\n)^{\tilde N}$. On the other hand, 
by~\cite[Prop.\,5.5]{CP}, the equality 
$\bb(\n)=\bb(\tilde\n)$ (see Proposition~\ref{prop:jos}) implies that 
$\gS(\tilde\n)^{\tilde N}\subset\gS(\n)$. Hence $\gS(\tilde\n)^{\tilde N}\subset\gS(\n)^{\tilde N}$.
\end{proof}

\begin{lm}    \label{lm:4.3}
For any nilradical $\n$, we have $\gS(\n)^{\tilde N}=\gS(\n)^U$.
\end{lm}
\begin{proof} 
{\bf 1}$^o$. Let us first prove that $\trdeg\gS(\n)^{\tilde N}=\trdeg\gS(\n)^U$.
Since $\tilde N$ is unipotent, the field of invariants $\BC(\n^*)^{\tilde N}$ is the quotient field of the
algebra of invariants $\BC[\n^*]^{\tilde N}=\gS(\n)^{\tilde N}$~\cite[Chap.\,1.4]{brion}. Furthermore, by Rosenlicht's theorem~\cite[Chap.\,1.6]{brion}, one has
\[
 \trdeg\BC(\n^*)^{\tilde N}+\max_{\xi\in\n^*}\dim \tilde N{\cdot}\xi=\dim\n ,
\] 
and the same holds for $U$ in place of $\tilde N$. Therefore, since 
\[
     \max_{\xi\in\n^*}\dim U{\cdot}\xi\ge \max_{\xi\in\n^*}\dim \tilde N{\cdot}\xi,
\] 
it suffices to prove that one has equality here. By Theorem~\ref{thm:dense-saturation}, $\tilde N{\cdot}\ce$ 
is dense in $\n^*$. Hence $U{\cdot}\ce$ is dense in $\n^*$, too. 
\\ \indent
Now, $\ut=(\ut\cap\tel)\oplus\tilde\n$, where $\tel$ is the standard Levi subalgebra of
$\tilde\p:={\sf norm}_\g(\tilde\n)$. Since $\tilde\n$ is optimal, $\ut\cap\tel$ stabilises any $\xi\in\ce_0\subset\ce$. Hence $\ut^\xi=(\ut\cap\tel)\oplus\tilde\n^\xi$ and 
$\dim U{\cdot}\xi=\dim \tilde N{\cdot}\xi$. Therefore, $\max_{\xi\in\n^*}\dim U{\cdot}\xi= \max_{\xi\in\n^*}\dim \tilde N{\cdot}\xi$, and we are done.

{\bf 2}$^o$. By the first part, the extension $\BC[\n^*]^{\tilde N}=\gS(\n)^{\tilde N}\subset\gS(\n)^U$ is algebraic. But the algebra of invariants of a {\bf connected} algebraic group $Q$ acting on an affine variety $X$ is algebraically closed in $\BC[X]$, see~\cite[p.\,100]{kr84}.
\end{proof}

Combining previous lemmas yields our main result on diagram~\eqref{eq:diagr}.
\begin{thm}     \label{thm:main5}
The four algebras of invariants that form the square in~\eqref{eq:diagr} are equal, i.e., 
\[
      \gS(\n)^U=\gS(\tilde\n)^U=\gS(\n)^{\tilde N}=\gS(\tilde \n)^{\tilde N} .
\]
These algebras are polynomial and their common transcendence degree is $\#\eus K(\n)=\ind\tilde\n$.
If $\n\ne\tilde\n$, then 
$\trdeg \gS(\n)^{N}=\trdeg \gS(\n)^{\tilde N}+ \dim(\tilde\n/\n)>\trdeg \gS(\n)^{\tilde N}$. Thus, there are at 
most two different algebras of invariants in~\eqref{eq:diagr}. 
\end{thm}
\begin{proof}
{\it\bfseries (1)}  By  Lemma~\ref{lm:4.1}, these four algebras are polynomial. 

{\it\bfseries  (2)} By  Lemmas~\ref{lm:4.2} and \ref{lm:4.3}, these four algebras are equal.
(Note that Lemma~\ref{lm:4.3} applies also to $\tilde\n$ in place of $\n$.) 

{\it\bfseries  (3)} Since both $\tilde N$ and $N$ are unipotent, it follows from the Rosenlicht theorem that
$\trdeg \gS(\tilde\n)^{\tilde N}=\ind\tilde\n$ and 
$\trdeg \gS(\n)^{N}=\ind\n$. Hence the last relation is just Eq.~\eqref{eq:compare-ind}.
\end{proof}

\begin{rmk}   \label{rem:arb-b-ideal-FA}
The algebra of $U$-invariants in $\gS(\rr)$ is polynomial  for an arbitrary {$\be$-stable} ideal $\rr$ of 
$\ut$. That is, if $\rr\subset\ut$ and $[\be,\rr]\subset\rr$, then $\gS(\rr)^U$ is a polynomial algebra. The 
reason is that $B$ has an open orbit in $\rr^*$, see \cite[Section\,4]{p17} for details.
\end{rmk}

\begin{cl} For any nilradical $\n$ and a cascade point $\bxi\in\ce_0\subset \n^*$, we have
\begin{enumerate}
\item \ $\# \eus K(\n)=\#\{\text{the divisors in } \ \n^*\setminus B{\cdot}\bxi\}
   =\#\{\text{the divisors in } \ \n^*\setminus F_\n{\cdot}\bxi\}$;
\item \ $B{\cdot}\bxi=F_\n{\cdot}\bxi$.
\end{enumerate}
\end{cl}
\begin{proof}
(1) \ It is known that 
\begin{itemize}
\item \ $\# \eus K(\n)=\ind\tilde\n=\trdeg\gS(\tilde\n)^{\tilde N}$;
\item \ $\#\{\text{the divisors in } \ \n^*\setminus B{\cdot}\bxi\}=\trdeg \gS(\n)^U$~\cite[Theorem\,4.4]{p17};
\item \ $\#\{\text{the divisors in } \ \n^*\setminus F_\n{\cdot}\bxi\}=\trdeg \gS(\n)^{\tilde N}$.
\end{itemize}
The last equality relies on the facts that the orbit $F_\n{\cdot}\bxi$ is open in $\n^*$ and $\tilde N$ is the unipotent radical of $F_\n$. Hence~\cite[Theorem\,4.4]{p17} applies also in this case.

(2) \ Since $B$ and $F_\n$ are solvable, the orbits $B{\cdot}\bxi$ and $F_\n{\cdot}\bxi$ are
affine. Therefore, both $\n^*\setminus B{\cdot}\bxi$ and $\n^*\setminus F_\n{\cdot}\bxi$ are the union
of divisors and the assertion follows from (1).
\end{proof}

If $\n$ is not optimal, then  
$\gS(\n)^N$ is not always polynomial, see Example~\ref{ex:non-polynom}. Actually, there are 
Lie algebras $\q$ such that $\gS(\q)^Q$ is not finitely generated! A construction of such $\q$ utilising the 
Nagata counterexample to Hilbert's 14th problem is given by J.\,Dixmier in~\cite[4.9.20(c)]{dix}. Nevertheless, we demonstrate 
below that, for the nilradicals in $\g$, the algebra $\gS(\n)^N$ is as good as the algebras of invariants of 
linear actions of reductive groups. Recall that a commutative associative $\BC$-algebra $\ca$ equipped 
with action of an algebraic group $Q$ is called a {\it rational} $Q$-agebra if, for any $a\in\ca$, the linear 
span in $\ca$ of the orbit $Q{\cdot}a$ is finite-dimensional and the representation of $Q$ on 
$\lg Q{\cdot}a\rg$ is rational (see e.g.~\cite[p.1]{gr97}).

\begin{thm}          \label{thm:inv-fin-generated}
Let $\n$ be an arbitrary nilradical in $\g$.  Then
\begin{itemize}
\item[\sf (i)] \ the algebra $\gS(\n)^N=\BC[\n^*]^N$ is finitely generated;
\item[\sf (ii)] \  the quotient variety $\n^*\md N$ has rational singularities.  
\end{itemize}
\end{thm}
\begin{proof}
{\sf (i)} \ Let $\p$ be the parabolic subalgebra with $\n=\p^{\sf nil}$. Let $\el$ be the standard Levi 
subalgebra of $\p$ and $L$ the corresponding Levi subgroup of $P={\sf Norm}_G(\n)\subset G$. Since
$[\el,\n]\subset \n$, the algebra $\gS(\n)^N$ is a rational $L$-algebra. 
Set $U(L)=U\cap L$. It is a maximal unipotent subgroup of $L$, and $U$ is a semi-direct product of 
$U(L)$ and $N$. By Theorem~\ref{thm:main5}, $\gS(\n)^U=(\gS(\n)^N)^{U(L)}$ is a polynomial algebra. 
Now, we can apply Corollary~4 from~\cite[\S\,3]{po86}, which asserts that if 
$\ca$ is a rational $L$-algebra, then $\ca$ is finitely generated if 
and only if $\ca^{U(L)}$ is. In our case $\ca=\gS(\n)^N$ and $\ca^{U(L)}$ is a polynomial algebra,
with $\#\eus K(\n)$ generators. Hence $\gS(\n)^N$ is finitely generated.

{\sf (ii)} \ Let {\it\bfseries (P)} be a so-called ``stable property'' of local rings of algebraic 
varieties (see~\cite[\S\,6]{po86} for the details). By~\cite[Theorem\,6]{po86}, $\ca$ has property 
{\it\bfseries (P)} if and only if $\ca^{U(L)}$ has. Since rationality of singularities is such a property, we 
obtain the second assertion.
\end{proof}

\begin{rmk}    \label{rem:jos}
{\sf (1)} \ Part~{\sf (i)} of Theorem~\ref{thm:inv-fin-generated} appears in~\cite[Lemma~4.6(ii)]{jos77}, with another proof. Namely, in place of reference to~\cite{po86}, one can provide the following simple argument for the implication
\\[.5ex]
\centerline{{\it $\ca^{U(L)}$ is finitely generated $\Longrightarrow$ $\ca$ is finitely generated}:}

Let $f_1,\dots,f_l$ be a set of generators for $\ca^{U(L)}$. We may also assume that each $f_i$ is a
$T$-eigenvector. Then the $\BC$-linear span of $L{\cdot}f_i\subset\ca$ is a simple finite-dimensional $L$-module, 
say $\mathsf{V}_i$. It is easily seen that the $\BC$-algebra generated by $\sum_{i=1}^l \mathsf{V}_i$ is 
$L$-stable and contains all simple $L$-modules from $\ca$. Hence the finite-dimensional space
$\sum_{i=1}^l \mathsf{V}_i$ generates $\ca$.

{\sf (2)} \ There is also an alternate approach to part~{\sf (ii)}. By a result of Kostant, the algebra 
$\BC[\n^*]^N$ is a multiplicity free $L$-module, see~\cite[4.5]{jos77}. Being the algebra of invariants of a 
linear action, $\BC[\n^*]^N$ is also integrally closed. Therefore, $\n^*\md N$ is a spherical $L$-variety.
By~\cite[Theorem\,10]{po86}, this implies that $\n^*\md N$ has rational singularities.
\end{rmk}

\begin{rmk}    \label{rem:popov+brion}
{\sf (1)} \ The general treatment of ``stable properties'' is due to V.\,Popov~\cite{po86}, but the assertion 
that relates rationality of singularities for $\ca$ and $\ca^{U(L)}$ goes back to H.\,Kraft and D.\,Luna, 
see exposition in Brion's thesis~\cite[Chap.\,I(c)]{br81}. It is worth noting that, for a rational $L$-algebra $\ca$, the equivalence:
\\[.6ex]
\centerline{\it $\ca$ is finitely generated $\Longleftrightarrow$ $\ca^{U(L)}$ is finitely generated}
holds over an algebraically closed field of an arbitrary characteristic, see~\cite[Theorem\,16.2]{gr97}.

{\sf (2)} \ A list of ``stable properties'' is found in~\cite[Section\,6]{po86}. 
\end{rmk}

\begin{ex}   \label{ex:non-polynom}
{\sf (1)} \ Take $\g=\mathfrak{sl}_5$ and $\n=\n_{\eus T}$ with ${\eus T}=\{\ap_3,\ap_4\}$. Then 
$\dim\n=7$, $\eus K(\n)=\eus K=\{\beta_1,\beta_2\}$, and $\tilde\n=\ut$. Hence $\ind\n=5$. Let 
$\{e_{ij}\mid 1\le i<j\le 5, j=4,5\}$ be the matrix units corresponding to $\n$ (i.e., a basis for $\n$). We
regard them as (linear) functions on $\n^-\simeq\n^*$. Here $e_{15}\in \g_{\beta_1}$ and
$e_{24}\in \g_{\beta_2}$. Obviously, $e_{15}$ and 
$f_{12}=\begin{vmatrix} e_{14} & e_{15} \\ e_{24} & e_{25} \end{vmatrix}$ belong to $\gS(\n)^U$. For the 
standard Levi subalgebra $\el$ of $\p={\sf norm}_\g(\n)$, one has $[\el,\el]=\sltri$, hence $\gS(\n)^N$ is 
an $\sltri$-module. Taking the $\sltri$-modules generated by $e_{15}$ and $f_{12}$, one obtains six functions that generate $\gS(\n)^N$: 
\[
  e_{15},e_{25},e_{35},f_{12}=\begin{vmatrix} e_{14} & e_{15} \\ e_{24} & e_{25} \end{vmatrix}, \ 
 f_{13}= \begin{vmatrix} e_{14} & e_{15} \\ e_{34} & e_{35} \end{vmatrix}, \ 
  f_{23}=\begin{vmatrix} e_{24} & e_{25} \\ e_{34} & e_{35} \end{vmatrix} .
\]
Since $f_{12}\not\in \BC[e_{15},e_{25},e_{35}]$ and a minimal generating system can be chosen to be 
$SL_3$-invariant, the set of generators above is minimal. These generators satisfy the relation 
$e_{15}f_{23}-e_{25}f_{13}+e_{35}f_{12}=0$, hence $\n^*\md N$ is a hypersurface. Under passage to 
$\tilde N$ or $U$, the situation improves, because $\gS(\n)^{\tilde N}=\gS(\n)^U=\BC[e_{15}, f_{12}]$.

{\sf (2)} \ More generally, for $\g=\mathfrak{sl}_N$ and $\eus T=\{\ap_{N-k},\dots,\ap_{N-1}\}$ with 
$k\le N/2$, one has 
\[  
\dim\n_\eus T=\frac{k}{2}(2N-1-k), \quad \eus K_\eus T=\{\beta_1,\dots,\beta_k\}, \text{ and }\ 
\ind\n_\eus T=\frac{k}{2}(2N+1-3k).
\] 
Here $[\el,\el]=\mathfrak{sl}_{N-k}$ and $\gS(\n_\eus T)^U=\BC[F_1,\dots, F_k]$, where $F_i$ is the ``anti-principal'' minor of degree $i$, i.e.,
\[
   F_1=e_{1,N}, \ F_2=\begin{vmatrix} e_{1,N-1} & e_{1,N} \\ e_{2,N-1} & e_{2,N} \end{vmatrix},\dots,
   F_k=\begin{vmatrix} e_{1,N-k+1} & \cdots & e_{1,N} \\ \cdots & \cdots & \cdots \\
   e_{k,N-k+1} & \cdots & e_{k,N} \end{vmatrix} .
\]
The $SL_{N{-}k}$-module generated by $F_i$, $\lg SL_{N-k}{\cdot}F_i\rg$, is isomorphic to 
$\wedge^i \mathbb V$, where $\mathbb V$ is the standard $SL_{N{-}k}$-module. It is also easily seen 
that 
$F_j$ does not belong to the $\BC$-algebra generated by $\sum_{i=1}^{j-1}\lg SL_{N-k}{\cdot}F_i\rg$.
Therefore, the minimal number of generators of $\eus Z(\n_\eus T)$ is 
\[
     M_\eus T=\sum_{j=1}^k \dim \wedge^i\mathbb V=\sum_{j=1}^k \genfrac{(}{)}{0pt}{}{N-k}{j} .
\] 
If $k=1$, then $\n_\eus T$ is abelian and $M_\eus T=\ind\n_\eus T=N-1$. If $k\ge 2$,
then $M_\eus T\ge \ind \n_\eus T$ and the equality holds only for $k=2, N=4$.

{\sf (3)} \ For $\GR{E}{6}$ and ${\eus T}=\{\ap_2\}$, one has $\dim\n_\eus T=25$,  
$\eus K_\eus T=\{\beta_1,\beta_2,\beta_3\}$, and $\ind\n_\eus T=13$. Here $[\el,\el]=
\tri\oplus\mathfrak{sl}_{5}$ and $\gS(\n_\eus T)^U=\BC[F_1,F_2,F_3]$, where $\deg F_i=1,2,4$ for
$i=1,2,3$, respectively. Using the $\el$-modules generated by $F_1$ and $F_2$, one can prove that $M_\eus T\ge 15$. 
\end{ex}

\section{On the square integrable nilradicals}
\label{sect:square-int}

\noindent
Recall that any nilradical $\n=\n_{\eus T}$ is equipped with the canonical $\BZ$-grading 
$\n=\bigoplus_{i=1}^{d_{\eus T}}\n(i)$, where $d_{\eus T}=\sum_{\ap\in {\eus T}}[\theta:\ap]$ and 
$\n(d_{\eus T})=\z(\n)$ is the centre of $\n$. Accordingly, one obtains the partition 
$\Delta(\n_{\eus T})=\bigsqcup_{i=1}^{d_{\eus T}}\Delta_{\eus T}(i)$, see Section~\ref{subs:optim}.

Clearly, $\ind\n\ge\dim\z(\n)$ and $\gS(\z(\n))\subset \gS(\n)^N$. Following~\cite{ag85, ag03}, 
we say that $\n$ is {\it square integrable}, if $\ind\n=\dim\z(\n)$. Then 
$\n^\xi=\z(\n)$ for any $\xi\in\n^*_{\sf reg}$ and hence $\cF(\n)=\z(\n)$ is abelian. Since
$\trdeg\gS(\n)^N=\dim\z(\n)$, the extension  $\gS(\z(\n))\subset \gS(\n)^N$ is algebraic. This 
implies that $\gS(\n)^N=\gS(\z(\n))$ and hence $\gS(\n)$ is a free $\gS(\n)^N$-module.
It is easily seen that a Heisenberg  
algebra $\h$ is square integrable, with $\ind\h=1$.
Using $\eus K(\n)$ and our description of $\n^\bxi$, we classify all square integrable nilradicals. 
Recall that $\eus K_{\eus T}=\eus K(\n_{\eus T})$ is a poset.

\begin{lm}    \label{lm:beta-2}
 For any $\n=\n_{\eus T}$, we have $\min\eus K_{\eus T}\subset\bigl(\Delta_{\eus T}(1)\cup \Delta_{\eus T}(2)\bigr)$.
\end{lm}
\begin{proof}
 Take $\beta\in \min \eus K_{\eus T}$ and $\ap\in\Phi(\beta)\cap {\eus T}$. Then there are the following possibilities:
 \begin{itemize}
\item $\beta=\ap$. Then $\beta\in\Delta_{\eus T}(1)$;
\item $\beta\ne\ap$ and $\Phi(\beta)=\{ \ap \}$. Recall that $\beta$ is the highest root in the root system with basis $\supp(\beta)$. Since $\beta$ is a minimal element of $\eus K_{\eus T}$, we have
$\supp(\beta)\cap {\eus T}=\{\ap\}$. Then $[\beta:\ap]=2$ and $\beta\in\Delta_{\eus T}(2)$;
\item $\Phi(\beta)=\{\ap,\ap'\}$. Here $\supp(\beta)$ is a basis for the root system $\GR{A}{p}$ 
($p\ge 2$) and $\supp(\beta)\cap {\eus T}=\{\ap,\ap'\}$. Then either $\ap'\in {\eus T}$ and 
$\beta\in\Delta_{\eus T}(2)$, or $\ap'\not\in {\eus T}$ and $\beta\in\Delta_{\eus T}(1)$.  \qedhere
\end{itemize}
\end{proof}

\begin{thm}    \label{thm:square-int}
Let $\n=\n_{\eus T}$ be an arbitrary nilradical. Then \\[.6ex]
\centerline{$\n$ is square integrable $\Longleftrightarrow$ \ $d_{\eus T}\le 2$ and
$\eus K_{\eus T}\subset \Delta_{\eus T}(d_{\eus T})=\Delta(\z(\n))$.} 
\end{thm}
\begin{proof}
{\sf (1)} \ Let $\bxi\in\ce_0$ be a cascade point. Then $\n^\bxi\supset \bigoplus_{\beta\in\eus K_{\eus T}}\g_\beta$.
If $\n$ is square integrable, then $\n^\bxi\subset \n(d_{\eus T})$. Hence $\eus K_{\eus T}\subset \Delta_{\eus T}(d_{\eus T})$. By Lemma~\ref{lm:beta-2}, this is only possible, if $d_{\eus T}\le 2$. Thus,
the implication ``$\Rightarrow$" is proved.

{\sf (2)} \  If $d_{\eus T}=1$, then the assumption on $\eus K_{\eus T}$ is vacuous and $\n$ is abelian, hence square integrable. Therefore, to prove the implication
``$\Leftarrow$", we may assume that $d_{\eus T}=2$. Then $\z(\n)=\n(2)=[\n,\n]$ and 
$\eus K_{\eus T}\subset\Delta_{\eus T}(2)$. Now, there are two possibilities for ${\eus T}$.

\textbullet \ \ Suppose that ${\eus T}=\{\ap\}$ and $[\theta:\ap]=2$. \\
Set $\beta_i=\Phi^{-1}(\ap)$. Since $\beta_i\in\eus K_{\eus T}\subset\Delta_{\eus T}(2)$, we have $[\beta_i:\ap]=2$. Let us prove that
$\beta_i$ is the unique minimal element of $\eus K_{\eus T}$.
Assume that $\beta_i\not\in\min \eus K_{\eus T}$, i.e., there is $\beta'\in \eus K_{\eus T}$
such that $\beta'\prec \beta$. Then $\supp(\beta')\subset\supp(\beta_i)\setminus\{\ap\}$ and
$\supp(\beta')\cap {\eus T}\ne \varnothing$. A contradiction! Hence $\beta_i$ is minimal in 
$\eus K_{\eus T}$. Assume that there is another minimal element of $\eus K_{\eus T}$, say $\tilde\beta$. Then 
$\supp(\beta_i)$ and $\supp(\tilde\beta)$ are disjoint, and we must have $\supp(\tilde\beta)\cap {\eus T}\ne \varnothing$, which is impossible.
This contradiction shows that $\min\eus K_{\eus T}=\{\beta_i\}$.

Since $B{\cdot}\bxi$ is dense in $\n^*$, $B{\cdot}\bxi\subset \n^*_{\sf reg}$, and $\n(2)$ is $B$-stable, it 
suffices to prove that $\n^\bxi\subset\n(2)$ (and then, actually, $\n^\bxi=\n(2)$).  Recall that 
$\n=\bigoplus_{\beta_j\in\eus K_{\eus T}}(\h_j\cap\n)$. For the unique minimal element $\beta_i\in\eus K_{\eus T}$, 
we have $\h_i\subset\n$ and hence $\h_i\cap\n^\bxi=\g_{\beta_i}$. If $\beta_i\prec\beta_j$, then
$[\beta_j:\ap]=2$ and the contribution from $\h_j\cap\n$ to $\n^\bxi$ is
\[
     \h_j\cap\n^\bxi=\g_{\beta_j}\oplus\bigl(\bigoplus_{\gamma\in\eus C_\n(j)}\g_{\beta_j-\gamma}\bigr) ,
\]
see Theorem~\ref{thm:cascade-point-stab}. Since $\eus K_{\eus T}\subset \Delta_{\eus T}(2)$, we have 
$\g_{\beta_j}\subset \n(2)$. The very
definition of $\eus C_{\n}(j)$ says that $\gamma\not\in\Delta_{\eus T}$. Hence $[\gamma:\ap]=0$ and 
$[\beta_j-\gamma:\ap]=2$. Thus, $\beta_j-\gamma\in \Delta_{\eus T}(2)$ and $\h_j\cap\n^\bxi\subset \n(2)$.
 
\textbullet \ \ Suppose that ${\eus T}=\{\ap,\ap'\}$ and $[\theta:\ap]=[\theta:\ap']=1$. \\ 
The argument here is similar to that in the previous part. Set $\beta=\Phi^{-1}(\ap)$ and
$\beta'=\Phi^{-1}(\ap')$. Using the hypothesis that $\eus K_{\eus T}\subset\Delta_{\eus T}(2)$, one first 
proves that $\beta=\beta'$ (hence $\Phi(\beta)=\{\ap,\ap'\}$) and that $\beta$ is the unique minimal 
element of $\eus K_{\eus T}$. Then the use of Theorem~~\ref{thm:cascade-point-stab} allows us to check 
that $\n^\bxi\subset\n(2)$.
\end{proof}

Using Theorem~\ref{thm:square-int}, it is not hard to get the list of square integrable nilradicals in all 
simple Lie algebras. For $d_{\eus T}=1$, $\n_{\eus T}$ is abelian and these cases are well known. If 
$d_{\eus T}=2$, then 
$[\n_{\eus T}, [\n_{\eus T},\n_{\eus T}]]=0$. The 
condition that $\eus K_{\eus T}$ belongs to the highest graded component of $\n_{\eus T}$ is quite 
strong. Therefore, not all nilradicals with $d_{\eus T}=2$ are square integrable.

\begin{ex}   \label{ex:metabelian} 
(1) The  list of square integrable nilradicals with $d_{\eus T}=2$ is given below.

\begin{tabular}{ll}
\textbullet \ \ $\GR{A}{n}$, \ ${\eus T}=\{\ap_{k}, \ap_{n+1-k}\}$ with $2k< n+1$; & 
\textbullet \ \ $\GR{C}{n}$, \ ${\eus T}=\{\ap_{k}\}$ with $k < n$; \\
\textbullet \ \ $\GR{B}{n}$, \ ${\eus T}=\{\ap_{2k}\}$ with $2k\le n$; &
\textbullet \ \ $\GR{D}{n}$, \ ${\eus T}=\{\ap_{2k}\}$ with $k\le [n/2]-1$; \\
\textbullet \ \ $\GR{D}{2n+1}$, \ ${\eus T}=\{\ap_{2n},\ap_{2n+1}\}$; & 
\textbullet \ \ $\GR{E}{6}$, \ ${\eus T}=\{\ap_6\}$ or ${\eus T}=\{\ap_1,\ap_5\}$;  \\ 
\textbullet \  \ $\GR{E}{7}$, \ ${\eus T}=\{\ap_6\}$ or ${\eus T}=\{\ap_2\}$; 
& \textbullet \ \ $\GR{E}{8}$, \ ${\eus T}=\{\ap_1\}$ or ${\eus T}=\{\ap_7\}$;  \\
\textbullet \ \ $\GR{F}{4}$, \ ${\eus T}=\{\ap_4\}$ or ${\eus T}=\{\ap_1\}$; & \textbullet \ \ $\GR{G}{2}$, \ ${\eus T}=\{\ap_2\}$.
\end{tabular}
\\[.9ex]  \indent
(2) The list contains all the cases in which $\n=\h_1$ is the 
Heisenberg nilradicals. For 
the exceptional algebras, this corresponds to ${\eus T}$ indicated first; whereas for the classical series, this 
corresponds to the cases with $k=1$.

(3) Using~\cite{CP}, one verifies that $\n_{\eus T}$ has no \CP\  for $\GR{B}{n}$ with $k\ge 2$, 
$(\GR{E}{8}, \ap_7)$, and $(\GR{F}{4}, \ap_1)$.  

(4) For $\GR{D}{2n}$ and ${\eus T}=\{\ap_{2n-1},\ap_{2n}\}$, $\n_{\eus T}$  is not square integrable. Indeed, here $d_{\eus T}=2$, 
$\eus K_{\eus T}=\{\beta_1,\beta_3,\dots,\beta_{2n-3},\beta_{2n-1},\beta_{2n}\}$, and
$\min\eus K_{\eus T}$ contains two elements, $\ap_{2n}=\beta_{2n-1}$ and $\ap_{2n-1}=\beta_{2n}$.
(We use the notation of Appendix~\ref{sect:tables}.) Hence
$\beta_{2n-1},\beta_{2n}\in \Delta_{\eus T}(1)$, while the other elements of $\eus K_{\eus T}$ belong to 
$\Delta_{\eus T}(2)$. Here $\ind\n_{\eus T}=2n^2{-}3n{+}3$ and 
$\dim\z(\n_{\eus T})=2n^2{-}3n{+}1$.  
(Of course, there are other nilradicals with $d_\eus T=2$ that are not square integrable.)
\end{ex}

\begin{rmk}
A real nilpotent Lie group $Q$ has a unitary square integrable representation if and only if 
$\ind\q=\dim\z(\q)$~\cite[Theorem\,1]{MW}. For this reason, A.G.\,Elashvili applied the term ``square 
integrable" to the nilpotent Lie algebras $\q$ satisfying that equality, over an algebraically closed field of 
characteristic zero~\cite{ag85}, cf. also~\cite{ag03}. Afterwards, the theory of square integrable 
representations was extended to the setting of arbitrary Lie groups~\cite{d82}. The relevant notions 
are those of a {\it quasi-reductive\/} Lie group and a coadjoint orbit of {\it reductive type}, see~\cite{DKT}. 
Therefore, Theorem~\ref{thm:square-int} provides a classification of the quasi-reductive nilradicals in 
$\g$. Various results on quasi-reductive {\it seaweed} (=\,biparabolic) subalgebras of $\g$ are obtained 
in~\cite{bm, MY12}.
\end{rmk}

\section{When is the quotient morphism equidimensional?}
\label{sect:EQ-i-module}
\noindent
It is well known that if a graded polynomial algebra $\ca$ is a free module over a graded subalgebra
$\cb$, then $\cb$ is necessarily polynomial. Furthermore, if $\cb$ is a graded polynomial subalgebra of  $\ca$, then $\ca$ is a free $\cb$-module if and only if the induced morphism 
$\pi: \spe(\ca)\to \spe(\cb)$ is {\it equidimensional}, i.e., $\dim\pi^{-1}(\pi(x))=\dim \ca-\dim\cb$ for any 
$x\in \spe(\ca)$~\cite[Prop.\,17.29]{gerry}. By a theorem of Chevalley, dominant equidimensional 
morphisms are open. Therefore, in the graded situation $\pi$ is also onto, see~\cite[2.4]{VG}.

In this section, we point out certain nilradicals $\n$ in $\g$ such that $\ca=\gS(\n)=\BC[\n^*]$ is a free 
module over $\cb=\gS(\n)^U=\BC[\n^*]^U$. The last algebra is always polynomial (Lemma~\ref{lm:4.1}), 
hence our task is to guarantee that the quotient morphism 
\[
   \pi: \spe(\ca)=\n^*\to \n^*\md U=\spe(\cb)
\] 
is equidimensional. Since $\gS(\n)^U=\gS(\n)^{\tilde N}$ (Lemma~\ref{lm:4.3}), the {\bf optimal} nilradicals 
of such type provide examples, where $\gS(\n)$ is a free module over its Poisson  centre 
$\gZ(\n)=\gS(\n)^N$. We begin with a simple observation.

\begin{lm}       \label{lm:K(n)=2}
If \ $\#\eus K(\n)\le 3$, then $\gS(\n)$ is a free $\gS(\n)^U$-module.
\end{lm}
\begin{proof}
By Theorem~\ref{thm:main5}, we have $\n^*\md U\simeq\mathbb A^{\#\eus K(\n)}$. Since 
$e_\theta\in\gS(\n)^U$ is an element of degree~$1$, the hyperplane $\n^*_0=\{\xi\in\n^*\mid \xi(e_\theta)=0\}$ is $U$-stable and $\n^*_0\md U\simeq \mathbb A^s$, where $s=\#\eus K(\n)-1\le 2$.
By a result of M.\,Brion~\cite{br83} on invariants of unipotent groups, 
$\pi_0:\n^*_0\to \n^*_0\md U$ is equidimensional, and this implies that $\pi$ is equidimensional, too.
\end{proof}

\begin{thm}       \label{thm:EQ-main1}
Let $\nap$ be an abelian nilradical, i.e., $[\theta:\ap]=1$. Suppose that a nilradical $\n$ is contained 
between $\nap$ and its optimisation $\widetilde{\nap}$. Then
\begin{enumerate}
\item \ $\nap$ is a \CP-ideal of\/ $\n$;
\item \ $\gS(\n)$ is a free module over $\gS(\n)^U=\gS(\n)^{\tilde N}$.
\end{enumerate}
\end{thm}
\begin{proof}
(1) \ By Remark~\ref{rem:more}, we have $\dim\nap=\bb(\nap)=\bb(\n)$.

(2) \ Consider the commutative diagram
$\begin{array}{ccc}
\n^*& \overset{\vp}{\longrightarrow} & \nap^* \\  
\Big\downarrow\vcenter{\rlap{$\pi$}} & & \Big\downarrow\vcenter{\rlap{$\pi_{\{\ap\}}$}} \\
\n^*\md U & \stackrel{\vp\md U}{\longrightarrow} & \nap^*\md U 
\end{array}$ \quad 
corresponding to the $U$-equivariant inclusion $\nap\hookrightarrow \n$. Since $\n$ and $\nap$ 
have the same optimisation, Theorem~\ref{thm:main5} implies that $\gS(\n)^U=\gS(\nap)^U$, i.e., 
$\vp\md U$ is an isomorphism. By a general result on $U$-invariants for the abelian 
nilradicals~\cite[Theorem\,4.6]{aura}, $\gS(\nap)$ is a free module over $\gS(\nap)^U$. Hence the 
morphism $\pi_{\{\ap\}}$ is equidimensional and onto. Because the projection $\vp$ is also onto and 
equidimensional, $\pi$ must be equidimensional, too. Since $\n^*\md U$ is an affine space, the morphism 
$\pi$ is flat and $\BC[\n^*]$ is a free $\BC[\n^*]^U$-module. 
\end{proof}

\begin{cl}    \label{cor:nad-centrom}
For the optimal nilradical\/ $\tilde\n=\widetilde\nap$, the algebra $\gS(\tilde\n)$ is a free module over 
$\gZ(\tilde\n)$. 
\end{cl}

Theorem~\ref{thm:EQ-main1} implies interesting consequences for some types of simple Lie algebras.

\begin{prop}    \label{prop:EQ-sl}
For every nilradical $\n$ in $\slno$, there is an $\ap\in\Pi$ such that 
$\nap\subset\n\subset \widetilde{\nap}$. Therefore, $\gS(\n)$ is a free module over $\gS(\n)^U$ for any\/ $\n$. 
\end{prop}
\begin{proof}
We use the notation of Example~\ref{ex:sln}(1). Recall that $[\theta:\ap]=1$ for each $\ap\in\Pi$.
If $\n=\n_{\eus T}$ and $\min\eus K(\n)=\{\beta_k\}$ \ 
($k\le [(n+1)]/2$), then ${\eus T}\cap\{\ap_k,\ap_{n+1-k}\}\ne \varnothing$. Then
any $\ap$ in this intersection will do. (The construction of such a simple root $\ap$ is essentially 
contained in the proof of Theorem~6.2 in \cite{ag03}. For, then $\nap$ is a \CP-ideal of $\n$. But our 
approach that refers to $\eus K(\n)$ is shorter.)
\end{proof}

Of course, the first assertion of Proposition~\ref{prop:EQ-sl} appears to be true because {\bf all} simple 
roots of $\slno$ provide abelian nilradicals. Nevertheless, the second assertion can be proved for the 
nilradicals in $\spn$, although the proof becomes more involved.

\begin{prop}    \label{prop:EQ-sp}
For any nilradical $\n$ in $\spn$, $\gS(\n)$ is a free module over
$\gS(\n)^U$. 
\end{prop}
\begin{proof}
Here $\napn$ is the only abelian nilradical and $\eus K(\napn)=\eus K$. Hence $\widetilde{\napn}=\ut$.

(1) \  If $\n=\n_{\eus T}$ and $\ap_n\in {\eus T}$, then $\napn\subset \n$ and Theorem~\ref{thm:EQ-main1} applies. 
In particular, this shows that here $\gS(\ut)$ is a free module over $\gS(\ut)^U=\gZ(\ut)$.

(2) \ Suppose that $\ap_n\not\in {\eus T}$. Then $\eus K(\n)=\{\beta_1,\dots,\beta_k\}$ for some $k<n$, 
$\ap_k\in {\eus T}$,  and
$\ah:=\n\cap\napn$ is a \CP-ideal of $\n$. 
If $\napn$ is identified with the space $\widehat{\mathsf{Sym}}_n$ of $n\times n$ matrices that are 
symmetric w.r.t. the antidiagonal (see Example~\ref{ex:spn}), then 
\begin{center}  
\begin{tikzpicture}[scale= .35] 
\draw (0,0)  rectangle (5,5);

\shade[top color=yellow,bottom color=gray, draw, line width=1pt] (2,0) -- (5,0) -- (5,5) -- (0,5) -- (0,2) -- (2,2)--cycle ;
\draw[dashed,brown]  (0,0) -- (5,5) ;
\draw[dashed,darkblue]  (2,2) -- (2,5) ;
\draw[dashed,darkblue]  (2,2) -- (5,2) ;

\draw (1.1,-.7)  node {\footnotesize {\color{darkblue}$n{-}k$}} ;
\draw (3.6,-.7)  node {\footnotesize {\color{darkblue}$k$}} ;
\draw (-1.1,2.4)  node {$\ah=$} ;
\draw (8, 2.4) node {$\subset \napn$};
\end{tikzpicture}
\end{center}

\noindent
Since $\ah$ is a $\be$-stable ideal of $\ut$, $\gS(\ah)^U$ is a polynomial algebra (see 
Remark~\ref{rem:arb-b-ideal-FA}). Let us prove that $\gS(\ah)$ is a free $\gS(\ah)^U$-module, i.e., that 
$\bar\pi:\ah^*\to \ah^*\md U$ is equidimensional. Consider the commutative diagram 
\begin{center}
$\begin{array}{ccc}
\napn^*& \overset{\psi}{\longrightarrow} & \ah^* \\  
\Big\downarrow\vcenter{\rlap{$\pi_{\ap_n}$}} 
& & \Big\downarrow\vcenter{\rlap{$\bar\pi$}} \\
\napn^*\md U & \stackrel{\psi\md U}{\longrightarrow} & \ah^*\md U 
\end{array}$ 
\end{center}
corresponding to the $U$-equivariant inclusion $\ah\hookrightarrow \napn$. Here 
$\gS(\n)^U=\gS(\ah)^U$ (cf. Section~\ref{subs:role}(2)), i.e., 
$\dim\ah^*\md U=\dim \n^*\md U=\#\eus K(\n)=k$. Upon the identification of $\napn^*$ with 
$\napn^-=\left\{\begin{pmatrix} 0 & 0 \\ C & 0\end{pmatrix}\mid C=\hat C \right\}$ and thereby with the
dual space $\widehat{\mathsf{Sym}}^*_n$, the algebra $\BC[\napn^*]^U$ 
is freely generated by the principal minors of $C$. Let $f_i$ be the principal minor of degree $i$, see the 
figure \quad 
\raisebox{-8.3ex}{\begin{tikzpicture}[scale= .35]
\draw (-1.3,3) node {$C=$};
\draw (6.4,3) node {$\left. {\parbox{1pt}{\vspace{60\unitlength}}}
\right\}$};
\draw (7.3,3) node {\footnotesize $n$};

\draw (0,0)  rectangle (1,1);
\draw (0,0)  rectangle (2,2);
\draw (0,0)  rectangle (3,3);
\draw (0,0)  rectangle (4,4);
\draw (0,0)  rectangle (6,6);
\draw (2,-.85) node {$\underbrace%
{\mbox{\hspace{40\unitlength}}}_{\large i}$} ;
\draw[dashed,brown]  (0,0) -- (6,6) ;
\end{tikzpicture} }.
\quad Then
$
    \gS(\napn)^U=\BC[\napn^*]^U=\BC[f_1,\dots,f_n] 
$
and $\pi_{\ap_n}(M)=(f_1(M),\dots,f_n(M))$. On the other hand, $f_i\subset \gS(\ah)$ for $i\le k$ and 
hence $\gS(\ah)^U=\BC[f_1,\dots,f_k]$. Furthermore, if $\bar M=\psi(M)\in\ah^*$, then $f_i(\bar M)=f_i(M)$ for $i\le k$.
Then $\psi\md U$ takes $(f_1(M),\dots,f_n(M))$ to $(f_1(M),\dots,f_k(M))$ and thereby
$\psi\md U$ is surjective and equidimensional. Since we have already proved that $\pi_{\ap_n}$ is 
onto and equidimensional, this yields the same conclusion for $\bar\pi$.

Once it is proved that $\bar\pi$ is onto and equidimensional, an argument similar to that in 
Theorem~\ref{thm:EQ-main1} can be applied to the embedding $\ah\hookrightarrow\n$. Consider 
the diagram
\beq    \label{CP-diagr}
\begin{array}{ccc}
\n^*& \overset{\phi}{\longrightarrow} & \ah^* \\  
\Big\downarrow\vcenter{\rlap{$\pi$}} 
& & \Big\downarrow\vcenter{\rlap{$\bar\pi$}} \\
\n^*\md U & \stackrel{\phi\md U}{\longrightarrow} & \ah^*\md U 
\end{array}. 
\eeq
Since $\ah$ is a \CP-ideal of $\n$, we have $\gS(\n)^N\subset \gS(\ah)$~\cite[Prop.\,5.5]{CP}. Hence 
$\gS(\n)^N=\gS(\ah)^N$ and thereby $\gS(\n)^U=\gS(\ah)^U$, i.e., $\phi\md U$ is an isomorphism. This 
implies that $\pi$ is equidimensional, and we are done.
\end{proof}

The proofs of Propositions~\ref{prop:EQ-sl} and \ref{prop:EQ-sp} exploit the fact that every nilradical in
$\slno$ or $\spn$ has a \CP\ (and hence a \CP-ideal). Using Theorem~\ref{thm:main5}, we get
a joint corollary to these propositions:
\beq     \label{eq:optimal-free}
\text{\it if\/
$\n=\tilde\n$ is an optimal nilradical in $\slno$ or $\spn$, then $\gS(\tilde\n)$ is a free 
$\gZ(\tilde\n)$-module.}
\eeq
For $\n\ne\tilde\n$, one cannot even guarantee that $\gZ(\n)$ is a polynomial algebra, see 
Example~\ref{ex:non-polynom}. It might be interesting to characterise non-optimal nilradicals having a 
polynomial algebra $\gZ(\n)$. 

By~\cite{CP}, if $\n\subset\widetilde\nap$ with $[\theta:\ap]=1$, then $\n$ has a \CP.
However, to point out a \CP-ideal of $\n$, one needs some precautions:

(1) \ if $\Phi(\Phi^{-1}(\ap))=\{\ap\}$, then $\n\cap\nap$ is a \CP-ideal of $\n$;

(2) \ if $\Phi(\Phi^{-1}(\ap))=\{\ap,\ap'\}$ and $\g\ne\slno$, then at least one of $\n\cap\nap$ and
$\n\cap\n_{\ap'}$ is a \CP-ideal of $\n$;

(3) \ If $\g=\slno$, then one should choose the {\bf minimal} $\widetilde\nap$ containing $\n$. More 
precisely, since $\n\subset\widetilde\nap$, we have $\eus K(\n)\subset
\eus K(\nap)$. Here one has to take $\ap$ such that $\eus K(\n)=\eus K(\nap)$. Then item (2) applies. 

In any case, $\n\cap\nap$ is a \CP-ideal of $\n$ for a ``right'' choice of $\ap\in\Pi$, and the hypothesis 
that $\nap\subset\n$ is not required for the presence of \CP. That is, Theorem~\ref{thm:EQ-main1} does 
not apply to all nilradicals with \CP. Nevertheless, using a case-by-case argument, we can prove the following.

\begin{thm}           \label{thm:Eq-CP-all}
If\/ $\n=\n_\eus T$ has a \CP, then $\gS(\n)$ is a free $\gS(\n)^U$-module.
\end{thm}
\begin{proof}
By the preceding discussion, we may assume that $\n\subset\widetilde\nap$ and $\ah:=\n\cap\nap$ is 
a \CP-ideal of $\n$. As in the proof of Proposition~\ref{prop:EQ-sp}, one has the commutative 
diagram~\eqref{CP-diagr}, where $\phi\md U$ is an isomorphism. Therefore, it suffices to prove that 
$\bar{\pi}:\ah^*\to \ah^*\md U$ is equidimensional.  
Consider all simple Lie algebras having abelian nilradicals.

(1) The algebras $\slno$ and $\spn$ have been considered above.

(2) For $\sono$, the only abelian nilradical corresponds to $\ap_1=\esi_1-\esi_2$ and 
$\eus K(\n_{\ap_1})=\{\beta_1,\beta_2\}$. Therefore, $\#\eus K(\n)\le 2$ and Lemma~\ref{lm:K(n)=2} 
applies.

(3) For $\sone$, the abelian nilradicals correspond to $\ap_1,\ap_{n-1}$, and $\ap_n=\esi_{n-1}+\esi_n$.

\textbullet\quad For $\ap_1=\esi_1-\esi_2$, the situation is the same as for $\g=\sono$.

\textbullet\quad Since $\n_{\{\ap_{n-1}\}}\simeq \n_{\{\ap_{n}\}}$, we consider only the last possibility.
The case of $(\sone, \ap_n)$ is similar to $(\spn, \ap_n)$, which is considered in 
Proposition~\ref{prop:EQ-sp}. The distinction is that now $\napn$ consists of $n\times n$ matrices that 
are {\it skew-symmetric\/} {w.r.t.} the antidiagonal, and the basic $U$-invariants are pfaffians of the 
principal minors of even order. The embedding $\ah\subset\napn$ can be described by the figure in 
the proof of Proposition~\ref{prop:EQ-sp}, only now $k$ must be even and the matrices should be 
skew-symmetric {w.r.t.} the antidiagonal. 
\\ \indent
More precisely, if $n=2p+1$, then $\ap_n\not\in \eus K$ and $\Phi^{-1}(\ap_n)=\beta_{2p-1}$.
If $n=2p$, then $\ap_{n}=\beta_{2p-1}$. In both cases, 
$\eus K(\n_{\{\ap_n\}})=\{\beta_1,\beta_3,\dots,\beta_{2p-1}\}$. If $\n\subset\widetilde\napn$ and 
$\nap\not\subset\n$, then $\eus K(\n)=\{\beta_1,\beta_3,\dots,\beta_{2s-1}\}$, where $s<p=[n/2]$. Then
$\trdeg\gS(\ah)^U=s$ and  $\gS(\ah)^U$ is generated by the pfaffians of order $2,4,\dots,2s$. The fact that these pfaffians
form a regular sequence in $\gS(\ah)=\BC[\ah^*]$ can be proved as in Proposition~~\ref{prop:EQ-sp}
for $\spn$.

(4) For $\GR{E}{6}$, the abelian nilradicals correspond to $\ap_1$ and $\ap_5$. In both cases,
we have $\#\eus K(\n_{\{\ap_i\}})=2$ and hence Lemma~\ref{lm:K(n)=2} applies to any $\n\subset
\widetilde{\n_{\{\ap_i\}}}$, $i=1,5$.

(5) For $\GR{E}{7}$, the only abelian nilradical corresponds to $\ap_1$. Here 
$\eus K(\n_{\{\ap_1\}})=\{\beta_1,\beta_2,\beta_3\}$, see Appendix~\ref{sect:tables}. 
Therefore, Lemma~\ref{lm:K(n)=2} applies here.
\end{proof}

Our computations suggest that Theorem~\ref{thm:Eq-CP-all} holds in the general case. 

\begin{conj}
For any nilradical $\n$ in a simple Lie algebra $\g$, \ $\gS(\n)$ is a free\/ $\gS(\n)^U$-module. In 
particular, for any optimal nilradical $\tilde\n$, $\gS(\tilde\n)$ is a free\/ $\gZ(\tilde\n)$-module.
\end{conj}

So far, this conjecture is proved for {\sf (i)} the square integrable nilradicals 
(Section~\ref{sect:square-int}), {\sf (ii)} the nilradicals with \CP, {\sf (iii)} the series $\GR{A}{n}$ and 
$\GR{C}{n}$, and {\sf (iv)} if $\rk\g\le 3$. It is not hard to check it for $\GR{D}{4}$. Perhaps, the first 
step towards a general proof is to verify the conjecture for $\n=\ut$. Since $\ut$ has a \CP\ only for 
$\GR{A}{n}$ and $\GR{C}{n}$, some fresh ideas are necessary here.

\appendix
\section{The elements of $\eus K$}
\label{sect:tables}

\noindent
We list below the cascade roots (elements of $\eus K$) for all simple Lie algebras. 
The numbering of $\Pi=\{\ap_1,\dots,\ap_{\rk \g}\}$ follows~\cite[Table\,1]{VO} 
and, for roots of the classical Lie algebras, we use the standard $\esi$-notation. The numbering of 
cascade roots yields a linear extension of the poset $(\eus K,\curle)$, i.e., it is not unique unless
$\eus K$ is a chain. In all cases, $\beta_1=\theta$ and 
\[
   \Phi(\beta_i)=\{\ap\in\Pi\mid \beta_i-\ap\in \Delta^+\cup\{0\}\} . 
\]
In particular, if $\beta\in \eus K\cap\Pi$, then $\Phi(\beta)=\{\beta\}$ and $\beta$ is a minimal element of 
$\eus K$. Conversely, if $\beta$ is a minimal element of $\eus K$ and $\Phi(\beta)=\{\ap\}$, a sole simple 
root, then $\ap=\beta$. 

\noindent
{\it\bfseries The cascade elements for the classical Lie algebras}:
\begin{description}
\item[$\GR{A}{n}, n\ge 2$]  \ $\beta_i=\esi_i-\esi_{n+2-i}=\ap_i+\dots +\ap_{n+1-i}$ \ ($i=1,2,\dots,\left[\frac{n+1}{2}\right]$);
\item[$\GR{C}{n}, n\ge 1$]  \ $\beta_i=2\esi_i=2(\ap_i+\dots+\ap_{n-1})+\ap_n$ \ ($i=1,2,\dots,n-1$) and 
$\beta_n=2\esi_n=\ap_n$;
\item[$\GR{B}{2n}$, $\GR{D}{2n}$, $\GR{D}{2n+1}$ ($n\ge 2$)]  \ 
$\beta_{2i-1}=\esi_{2i-1}+\esi_{2i}$, $\beta_{2i}=\esi_{2i-1}{-}\esi_{2i}=\ap_{2i-1}$ \ ($i=1,2,\dots,n$);
\item[$\GR{B}{2n+1}, n\ge 1$] \ here $\beta_1,\dots,\beta_{2n}$ are as above and $\beta_{2n+1}=\esi_{2n+1}=\ap_{2n+1}$;
\end{description}

\noindent
For all orthogonal series, we have $\beta_{2i}=\ap_{2i-1}$, $i=1,\dots,n$, while formulae for 
$\beta_{2i-1}$ via $\Pi$ slightly differ for different series. E.g. for $\GR{D}{2n}$ one has
$\beta_{2i-1}=\ap_{2i-1}+2(\ap_{2i}+\dots+\ap_{2n-2})+\ap_{2n-1}+\ap_{2n}$ ($i=1,2,\dots,n-1$)
and $\beta_{2n-1}=\ap_{2n}$.

\noindent
{\it\bfseries The cascade elements for the exceptional Lie algebras}:
\begin{description}
\item[$\GR{G}{2}$]  \ $\beta_1=(32)=3\ap_1+2\ap_2, \ \beta_2=(10)=\ap_1$;
\item[$\GR{F}{4}$]   \ $\beta_1=(2432)=2\ap_1{+}4\ap_2{+}3\ap_3{+}2\ap_4,\ \beta_2=(2210),\ \beta_3=(0210),\ \beta_4=(0010)=\ap_3$;
     \item[$\GR{E}{6}$] \  
  $\beta_1=$\raisebox{-2.1ex}{\begin{tikzpicture}[scale= .85, transform shape]
\node (a) at (0,0) {1}; \node (b) at (.2,0) {2};
\node (c) at (.4,0) {3}; \node (d) at (.6,0) {2};
\node (e) at (.8,0) {1}; \node (f) at (.4,-.4) {2};
\end{tikzpicture}}\!\!, 
  $\beta_2=$\raisebox{-2.1ex}{\begin{tikzpicture}[scale= .85, transform shape]
\node (a) at (0,0) {1}; \node (b) at (.2,0) {1};
\node (c) at (.4,0) {1}; \node (d) at (.6,0) {1};
\node (e) at (.8,0) {1}; \node (f) at (.4,-.4) {0};
\end{tikzpicture}}, 
  $\beta_3=$\raisebox{-2.1ex}{\begin{tikzpicture}[scale= .85, transform shape]
\node (a) at (0,0) {0}; \node (b) at (.2,0) {1};
\node (c) at (.4,0) {1}; \node (d) at (.6,0) {1};
\node (e) at (.8,0) {0}; \node (f) at (.4,-.4) {0};
\end{tikzpicture}},
  $\beta_4$\,=\raisebox{-2.1ex}{\begin{tikzpicture}[scale= .85, transform shape]
\node (a) at (0,0) {0}; \node (b) at (.2,0) {0};
\node (c) at (.4,0) {1}; \node (d) at (.6,0) {0};
\node (e) at (.8,0) {0}; \node (f) at (.4,-.4) {0};
\end{tikzpicture}}=\,$\ap_3$;
     \item[$\GR{E}{7}$] \ 
  $\beta_1$=\raisebox{-2.1ex}{\begin{tikzpicture}[scale= .85, transform shape]
\node (a) at (0,0) {1}; \node (b) at (.2,0) {2}; \node (c) at (.4,0) {3};
\node (d) at (.6,0) {4}; \node (e) at (.8,0) {3}; \node (f) at (1,0) {2};
\node (g) at (.6,-.4) {2};
\end{tikzpicture}},
  $\beta_2$=\raisebox{-2.1ex}{\begin{tikzpicture}[scale= .85, transform shape]
\node (a) at (0,0) {1}; \node (b) at (.2,0) {2}; \node (c) at (.4,0) {2};
\node (d) at (.6,0) {2}; \node (e) at (.8,0) {1}; \node (f) at (1,0) {0};
\node  at (.6,-.4) {1};
\end{tikzpicture}},
  $\beta_3$=\raisebox{-2.1ex}{\begin{tikzpicture}[scale= .85, transform shape]
\node (a) at (0,0) {1}; \node (b) at (.2,0) {0}; \node (c) at (.4,0) {0};
\node (d) at (.6,0) {0}; \node (e) at (.8,0) {0}; \node (f) at (1,0) {0};
\node (g) at (.6,-.4) {0};
\end{tikzpicture}}=\,$\ap_1$,
  $\beta_4$=\raisebox{-2.1ex}{\begin{tikzpicture}[scale= .85, transform shape]
\node (a) at (0,0) {0}; \node (b) at (.2,0) {0}; \node (c) at (.4,0) {1};
\node (d) at (.6,0) {2}; \node (e) at (.8,0) {1}; \node (f) at (1,0) {0};
\node (g) at (.6,-.4) {1};
\end{tikzpicture}},
  $\beta_5$=\raisebox{-2.1ex}{\begin{tikzpicture}[scale= .85, transform shape]
\node (a) at (0,0) {0}; \node (b) at (.2,0) {0}; \node (c) at (.4,0) {1};
\node (d) at (.6,0) {0}; \node (e) at (.8,0) {0}; \node (f) at (1,0) {0};
\node (g) at (.6,-.4) {0};
\end{tikzpicture}}=\,$\ap_3$,   \\
  $\beta_6$=\raisebox{-2.1ex}{\begin{tikzpicture}[scale= .85, transform shape]
\node (a) at (0,0) {0}; \node (b) at (.2,0) {0}; \node (c) at (.4,0) {0};
\node (d) at (.6,0) {0}; \node (e) at (.8,0) {1}; \node (f) at (1,0) {0};
\node (g) at (.6,-.4) {0};
\end{tikzpicture}}=\,$\ap_5$,
  $\beta_7$=\raisebox{-2.1ex}{\begin{tikzpicture}[scale= .85, transform shape]
\node (a) at (0,0) {0}; \node (b) at (.2,0) {0}; \node (c) at (.4,0) {0};
\node (d) at (.6,0) {0}; \node (e) at (.8,0) {0}; \node (f) at (1,0) {0};
\node (g) at (.6,-.4) {1};
\end{tikzpicture}}=\,$\ap_7$;
      \item[$\GR{E}{8}$] \ 
  $\beta_1$=\raisebox{-2.1ex}{\begin{tikzpicture}[scale= .85, transform shape]
\node (a) at (0,0) {2}; \node (b) at (.2,0) {3}; \node (c) at (.4,0) {4};
\node (d) at (.6,0) {5}; \node (e) at (.8,0) {6}; \node (f) at (1,0) {4}; \node (g) at (1.2,0) {2};
\node (h) at (.8,-.4) {3};
\end{tikzpicture}},
  $\beta_2$=\raisebox{-2.1ex}{\begin{tikzpicture}[scale= .85, transform shape]
\node (h) at (-.2,-.0) {0}; \node (a) at (0,0) {1}; \node (b) at (.2,0) {2}; \node (c) at (.4,0) {3};
\node (d) at (.6,0) {4}; \node (e) at (.8,0) {3}; \node (f) at (1,0) {2};
\node (g) at (.6,-.4) {2};
\end{tikzpicture}},
  $\beta_3$=\raisebox{-2.1ex}{\begin{tikzpicture}[scale= .85, transform shape]
\node (h) at (-.2,-.0) {0}; \node (a) at (0,0) {1}; \node (b) at (.2,0) {2}; \node (c) at (.4,0) {2};
\node (d) at (.6,0) {2}; \node (e) at (.8,0) {1}; \node (f) at (1,0) {0};
\node  at (.6,-.4) {1};
\end{tikzpicture}},
  $\beta_4$=\raisebox{-2.1ex}{\begin{tikzpicture}[scale= .85, transform shape]
\node (h) at (-.2,-.0) {0}; \node (a) at (0,0) {1}; \node (b) at (.2,0) {0}; \node (c) at (.4,0) {0};
\node (d) at (.6,0) {0}; \node (e) at (.8,0) {0}; \node (f) at (1,0) {0};
\node (g) at (.6,-.4) {0};
\end{tikzpicture}}=\,$\ap_2$,
  $\beta_5$=\raisebox{-2.1ex}{\begin{tikzpicture}[scale= .85, transform shape]
\node (h) at (-.2,-.0) {0}; \node (a) at (0,0) {0}; \node (b) at (.2,0) {0}; \node (c) at (.4,0) {1};
\node (d) at (.6,0) {2}; \node (e) at (.8,0) {1}; \node (f) at (1,0) {0};
\node (g) at (.6,-.4) {1};
\end{tikzpicture}},  \\
  $\beta_6$=\raisebox{-2.1ex}{\begin{tikzpicture}[scale= .85, transform shape]
\node (h) at (-.2,-.0) {0}; \node (a) at (0,0) {0}; \node (b) at (.2,0) {0}; \node (c) at (.4,0) {1};
\node (d) at (.6,0) {0}; \node (e) at (.8,0) {0}; \node (f) at (1,0) {0};
\node (g) at (.6,-.4) {0};
\end{tikzpicture}}=\,$\ap_4$, 
  $\beta_7$=\raisebox{-2.1ex}{\begin{tikzpicture}[scale= .85, transform shape]
\node (h) at (-.2,-.0) {0}; \node (a) at (0,0) {0}; \node (b) at (.2,0) {0}; \node (c) at (.4,0) {0};
\node (d) at (.6,0) {0}; \node (e) at (.8,0) {1}; \node (f) at (1,0) {0};
\node (g) at (.6,-.4) {0};
\end{tikzpicture}}=\,$\ap_6$,
  $\beta_8$=\raisebox{-2.1ex}{\begin{tikzpicture}[scale= .85, transform shape]
\node (h) at (-.2,-.0) {0}; \node (a) at (0,0) {0}; \node (b) at (.2,0) {0}; \node (c) at (.4,0) {0};
\node (d) at (.6,0) {0}; \node (e) at (.8,0) {0}; \node (f) at (1,0) {0};
\node (g) at (.6,-.4) {1};
\end{tikzpicture}}=\,$\ap_8$;
\end{description}
For the reader's convenience, we provide the Hasse diagram of the cascade posets for $\GR{D}{n}$
and $\GR{E}{n}$. The node `$i$' in the diagram represents $\beta_i\in\eus K$, and we attach the set 
$\Phi(\beta_i)\subset\Pi$ to it.

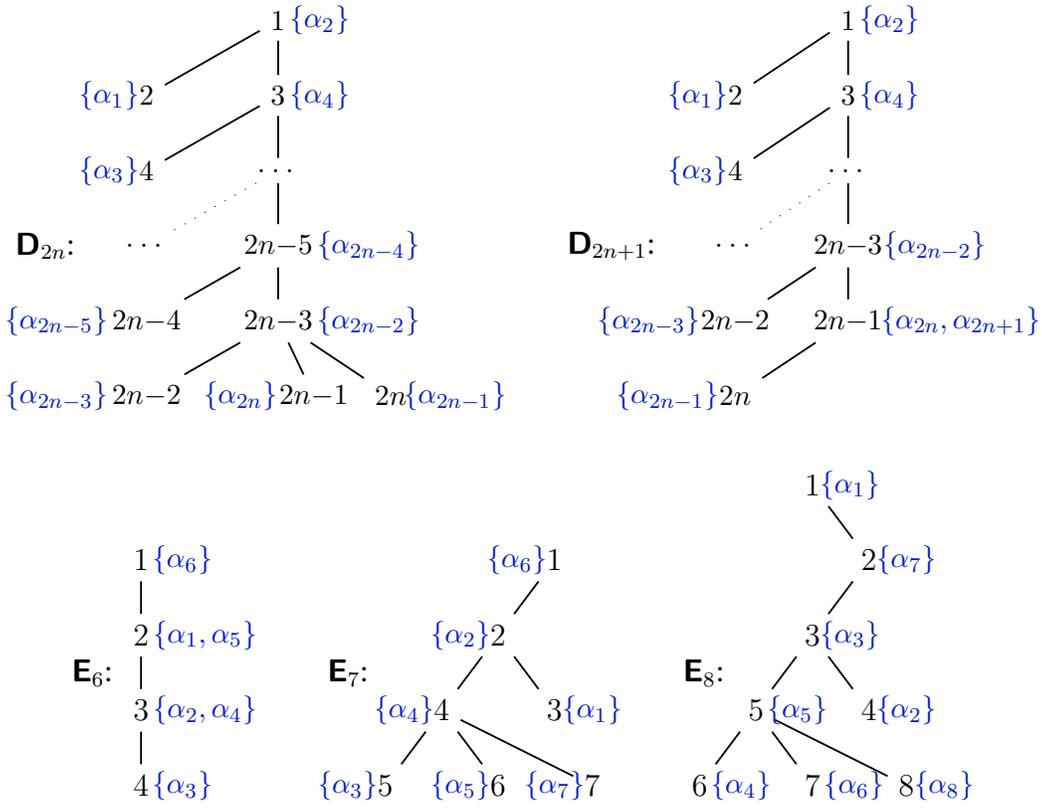
\begin{center}
\begin{tikzpicture}[scale= .5] 
\node  (1) at (9,10) {\small $1$};
\node  (2) at (5.5,8) {\small $2$};
\node  (3) at (9,8) {\small $3$};
\node  (4) at (5.5,6) {\small $4$};
\node  (5) at (9,6) {$\dots$};
\node  (6) at (5.5,4) {$\dots$};
\node  (7) at (9,4) {\small ${2n{-}5}$};
\node  (8) at (5.5,2) {\small ${2n{-}4}$};
\node  (9) at (9,2) {\small ${2n{-}3}$}; 
\node  (10) at (5.5,0) {\small ${2n{-}2}$}; 
\node  (11) at (9.95,0) {\small ${2n{-}1}$}; 
\node  (12) at (12,0) {\small ${2n}$}; 
\foreach \from/\to in {1/2, 1/3, 3/4, 3/5, 5/7,7/8, 7/9, 9/10, 9/11,9/12} \draw [-,line width=.7pt] (\from) -- (\to);
\draw[loosely dotted] (5)--(6);

\draw (10.1,10) node {{\color{darkblue}\small   $\{\ap_2\}$}};
\draw (4.5,8) node {{\color{darkblue}\small   $\{\ap_1\}$}};
\draw (10.1,8) node {{\color{darkblue}\small   $\{\ap_4\}$}};
\draw (4.5, 6) node {{\color{darkblue}\small   $\{\ap_3\}$}};
\draw (11.4, 4) node {{\color{darkblue}\small   $\{\ap_{2n-4}\}$}};
\draw (3.1,2) node {{\color{darkblue}\small   $\{\ap_{2n-5}\}$}};
\draw (11.4,2) node {{\color{darkblue}\small   $\{\ap_{2n-2}\}$}};
\draw (3.1,0) node {{\color{darkblue}\small   $\{\ap_{2n-3}\}$}};
\draw (8,0) node {{\color{darkblue}\small   $\{\ap_{2n}\}$}};
\draw (13.7,0) node {{\color{darkblue}\small   $\{\ap_{2n-1}\}$}};

\draw (2.8,4) node {$\GR{D}{2n}$:};
\end{tikzpicture}
\quad
\begin{tikzpicture}[scale= .5] 
\node (1) at (9,10) {\small $1$};
\node (2) at (6,8) {\small $2$};
\node (3) at (9,8) {\small $3$};
\node (4) at (6,6) {\small $4$};
\node (5) at (9,6) {$\dots$};
\node  (6) at (6,4)     {$\dots$};
\node (7) at (9,4) {\small ${2n{-}3}$};
\node (8) at (6,2) {\small ${2n{-}2}$};
\node (9) at (9,2) {\small ${2n{-}1}$}; 
\node (10) at (6,0) {\small ${2n}$}; 
\foreach \from/\to in {1/2, 1/3, 3/4, 3/5, 5/7,7/8, 7/9, 9/10} \draw [-,line width=.7pt] (\from) -- (\to);
\draw[loosely dotted] (5)--(6);

\draw (10.1,10) node {{\color{darkblue}\small   $\{\ap_2\}$}};
\draw (5,8) node {{\color{darkblue}\small   $\{\ap_1\}$}};
\draw (10.1,8) node {{\color{darkblue}\small   $\{\ap_4\}$}};
\draw (5, 6) node {{\color{darkblue}\small   $\{\ap_3\}$}};
\draw (11.3, 4) node {{\color{darkblue}\small   $\{\ap_{2n-2}\}$}};
\draw (3.7,2) node {{\color{darkblue}\small   $\{\ap_{2n-3}\}$}};
\draw (12,2) node {{\color{darkblue}\small   $\{\ap_{2n},\ap_{2n+1}\}$}};
\draw (4.2,0) node {{\color{darkblue}\small   $\{\ap_{2n-1}\}$}};

\draw (2.7,4) node {$\GR{D}{2n+1}$:};
\end{tikzpicture}
\end{center}

\vspace{1ex}
\begin{center}
\begin{tikzpicture}[scale= .5] 
\node (2) at (4.5,6) {$1$};
\node (3) at (4.5,4) {$2$};
\node (4) at (4.5,2) {$3$};
\node (5) at (4.5,0) {$4$};
\foreach \from/\to in {2/3, 3/4, 4/5} \draw [-,line width=.7pt] (\from) -- (\to);

\draw (5.6,6) node {{\color{darkblue}\small $\{\ap_6\}$}};
\draw (6.2,4) node {{\color{darkblue}\small  $\{\ap_1,\ap_5\}$}};
\draw (6.2,2) node {{\color{darkblue}\small  $\{\ap_2,\ap_4\}$}};
\draw (5.6,0) node {{\color{darkblue}\small  $\{\ap_3\}$}};

\draw (3.2,3) node {$\GR{E}{6}$:};
\end{tikzpicture}
\quad
\begin{tikzpicture}[scale= .5] 
\node (2) at (10.5,6) {$1$};
\node (3) at (9,4) {$2$};
\node (4) at (10.5,2) {$3$};
\node (5) at (7.5,2) {$4$};
\node (6) at (6,0) {$5$};
\node (7) at (9,0) {$6$};
\node (8) at (11.5,0) {$7$}; 
\foreach \from/\to in {2/3, 3/4, 3/5, 5/6, 5/7, 5/8} \draw [-,line width=.7pt] (\from) -- (\to);

\draw (9.5,6) node {{\color{darkblue}\small  $\{\ap_6\}$}};
\draw (8, 4) node {{\color{darkblue}\small  $\{\ap_2\}$}};
\draw (11.5, 2) node {{\color{darkblue}\small  $\{\ap_1\}$}};
\draw (6.5,2) node {{\color{darkblue}\small  $\{\ap_4\}$}};
\draw (5,0) node {{\color{darkblue}\small  $\{\ap_3\}$}};
\draw (8,0) node {{\color{darkblue}\small  $\{\ap_5\}$}};
\draw (10.5,0) node {{\color{darkblue}\small  $\{\ap_7\}$}};

\draw (5,3) node {$\GR{E}{7}$:};
\end{tikzpicture}
\quad
\begin{tikzpicture}[scale= .5] 
\node (1) at (9,8) {$1$};
\node (2) at (10.5,6) {$2$};
\node (3) at (9,4) {$3$};
\node (4) at (10.5,2) {$4$};
\node (5) at (7.5,2) {$5$};
\node (6) at (6,0) {$6$};
\node (7) at (9,0) {$7$};
\node (8) at (11.5,0) {$8$}; 
\foreach \from/\to in {1/2, 2/3, 3/4, 3/5, 5/6, 5/7, 5/8} \draw [-,line width=.7pt] (\from) -- (\to);

\draw (10,8) node {{\color{darkblue}\small  $\{\ap_1\}$}};
\draw (11.5,6) node {{\color{darkblue}\small  $\{\ap_7\}$}};
\draw (10, 4) node {{\color{darkblue}\small  $\{\ap_3\}$}};
\draw (11.5, 2) node {{\color{darkblue}\small  $\{\ap_2\}$}};
\draw (8.6,2) node {{\color{darkblue}\small  $\{\ap_5\}$}};
\draw (7.1,0) node {{\color{darkblue}\small  $\{\ap_4\}$}};
\draw (10.1,0) node {{\color{darkblue}\small  $\{\ap_6\}$}};
\draw (12.5,0) node {{\color{darkblue}\small  $\{\ap_8\}$}};

\draw (6.1,3) node {$\GR{E}{8}$:};
\end{tikzpicture}
\end{center}


\begin{thebibliography}{P95}
\bibitem{bm} {\sc K.~Baur} and {\sc A.~Moreau},
Quasi-reductive (bi)parabolic subalgebras of reductive Lie algebras, {\it Ann. Inst. Fourier}, 
{\bf 61}, no.\,2 (2011), 417--451.

\bibitem{br81}  {\sc M.~Brion}. 
Sur la th\`eorie des invariants,
{\it Publ. Math. Univ. Pierre et Marie Curie},  no.\,{\bf 45} (1981), pp. 1--92.

\bibitem{br83}  {\sc M.~Brion}. 
Surfaces quotients par un groupe unipotent,
{\it Comm. Algebra}, {\bf 11}, no.\,9 (1983), 1011--1014.

\bibitem{brion}  {\sc M.~Brion}. Invariants et covariants des groupes 
alg\'ebriques r\'eductifs, In: ``{\it Th\'eorie des invariants et g\'eometrie
des vari\'et\'es quotients}''
(Travaux en cours, t.\,{\bf 61}), 83--168,
Paris: Hermann, 2000.

\bibitem{dix} {\sc J.~Dixmier}.
{\it ``Enveloping Algebras''},  
AMS, Providence, RI, 1996. xx+379 pp.

\bibitem{d82} {\sc M.~Duflo}.
Th\'{e}orie de Mackey pour les groupes de Lie alg\'{e}briques, 
{\it Acta Math.} {\bf 149}\,(1982), 153--213.

\bibitem{DKT} {\sc M.~Duflo, M.S.~Khalgui} and {\sc P.~ Torasso}.
Alg\`{e}bres de Lie quasi-r\'{e}ductives,  
{\it Transform. Groups} {\bf 17}, no.\,2, (2012), 417--470.
 
\bibitem{ag72} 
{\sc A.G.~Elashvili}.
Canonical form and stationary subalgebras of points of 
general position for simple linear Lie groups, {\it Funct. Anal. Appl.},
{\bf 6}\,(1972), 44--53.

\bibitem{ag85} {\sc A.G.~Elashvili}.
 On the index of horospherical subalgebras of semisimple Lie algebras, {\it Trudy Razmadze
Math. Institute} (Tbilisi), {\bf 77}\,(1985), 116--126 (Russian). \  (MR0862923, Zbl 0626.17004)

\bibitem{ag03} {\sc A.G.~Elashvili} and {\sc A.~Ooms}. 
On commutative polarizations, {\it J.~Algebra},  {\bf 264}\,(2003), 129--154.

\bibitem{gr97} {\sc F.~Grosshans}. ``Algebraic homogeneous spaces and
invariant theory'', Lect. Notes Math. {\bf 1673}, Berlin: Springer, 1997.

\bibitem{jos77} {\sc A.~Joseph}. 
A preparation theorem for the prime spectrum of a semisimple Lie algebra,
{\it  J. Algebra}  {\bf 48}\,(1977), 241--289.

\bibitem{ko12} {\sc B.~Kostant}. 
The cascade of orthogonal roots and the coadjoint structure of the nilradical of a Borel subgroup of a semisimple Lie group, {\it Mosc. Math. J.}, {\bf 12}\,(2012), no.~3, 605--620.

\bibitem{ko13} {\sc B.~Kostant}. 
Center $\eus U(\n)$, cascade of orthogonal roots, and a construction of Lipsman--Wolf, ``Lie groups: structure, actions, and representations'', 163--173, Progr. Math., {\bf 306}, Birkh\"auser/Springer, New York, 2013.

\bibitem{kr84} {\sc H.~Kraft}. 
``{\it Geometrische Methoden in der Invariantentheorie\/}'', Aspekte der Mathematik {\bf D1}, 
Braunschweig: Vieweg \& Sohn, 1984.

\bibitem{MW} {\sc C.C.~Moore} and {\sc J.A.~Wolf}. 
Square integrable representations of nilpotent groups, 
{\it Trans. Amer. Math. Soc.}, {\bf 185}\,(1973), 445--462.

\bibitem{MY12} {\sc A.~Moreau} and {\sc O.~Yakimova}.
Coadjoint orbits of reductive type of parabolic and seaweed Lie subalgebras,
{\it Int. Math. Res. Not.}, {\bf 2012}, no.\,19, 4475--4519 (2012).

\bibitem{VO}  {\sc A.L.~Onishchik} and 
{\sc E.B.~Vinberg}. ``Lie groups and algebraic groups'', Berlin: Springer, 1990.

\bibitem{ooms}  {\sc A.~Ooms}.
On certain maximal subfields in the quotient division ring of an enveloping algebra, 
{\it J. Algebra}, {\bf  230}\,(2000), no.~2, 694--712.

\bibitem{ooms2}  {\sc A.~Ooms}.
The Frobenius semiradical of a Lie algebra, 
{\it J. Algebra}, {\bf 273}\,(2004), no.~1, 274--287.

\bibitem{aura} {\sc D.~Panyushev}.
Parabolic subgroups with Abelian unipotent radical as a testing site
for Invariant Theory, {\it Canad. J. Math.}, {\bf 51}\,(1999), no.~3, 616--635.

\bibitem{p17}  {\sc D.~Panyushev}. 
On the orbits of a Borel subgroup in abelian ideals, 
{\it Transform. Groups}, {\bf 22}, no.\,2 (2017), 503--524. 

\bibitem{CP}  {\sc D.~Panyushev}. Commutative polarisations and the Kostant cascade,
{\it Algebr. Represent. Theory}, {\bf 26}, no.\,3 (2023), 967--985.

\bibitem{po86} {\sc V.L.~Popov}. Contractions of actions of reductive algebraic
groups, {\it Math. USSR-Sb.}, {\bf 58}\,(1987), 311--335.

\bibitem{r72} {\sc R.W.~Richardson}. 
Principal orbit types for algebraic transformation spaces in characteristic zero, 
{\it Invent. Math.}, {\bf 16}\,(1972), 6--14.

\bibitem{sad}  {\sc S.T.~Sadetov}. A proof of the Mishchenko--Fomenko conjecture, 
{\it Doklady Math.}, {\bf 70}\,(2004), no.~1, 634--638.

\bibitem{gerry}  {\sc G.~Schwarz}. 
Lifting smooth homotopies of orbit spaces,
{\it Publ. Math. I.H.E.S.}, {\bf 51}\,(1980), 37--135.

\bibitem{tayu1} {\sc P.~Tauvel} and  {\sc R.~Yu}.
Indice et formes lin\'eaires stables dans les alg\`ebres de Lie,
{\it J.\,Algebra}, {\bf 273}\,(2004), 507--516.

\bibitem{vi90} 
{\sc E.B.\,Vinberg.} On certain commutative subalgebras of a universal enveloping algebra, 
{\it Math. USSR-Izv.}, {\bf 36}\,(1991), 1--22.

\bibitem{VG}  {\sc E.B.~Vinberg} and {\sc S.G.~Gindikin}. Degeneration of horospheres in 
spherical homogeneous spaces, {\it Funct. Anal. Appl.}, {\bf 52}\,(2018), 83--92.  

\end{thebibliography}
\end{document}